\documentclass[a4paper, inner=3cm, outer=3cm, top=3cm, bottom=3cm, 12pt]{article}

\usepackage{latexsym,amssymb,amsmath,epsfig,amsfonts,graphicx,mathrsfs}

\usepackage{psfrag,epsf}
\usepackage{subcaption}
\usepackage{enumerate}
\usepackage{dsfont}
\usepackage{geometry}
\usepackage{array}

\usepackage{amsthm}
\usepackage{color}
\usepackage{bm}

\usepackage{booktabs}

\newcommand{\vc}[1]{\boldsymbol{#1}}

\usepackage{dsfont}

\def\F{{\mathcal F}}

\def\Natural{{\mathbb{N}}}
\def\Real{\mathbb{R}}
\newcommand{\D}[1]{\mathrm{d}{#1}}
\def\E{{ \mathsf E}}
\def\var{{\mathsf{Var}}}
\def\cov{{\mathsf{Cov}}}
\def\logL{{\ell}}
\def\logLp{{\ell_{\rm p}}}
\def\pr{{\mathsf P}}
\def\v{{\epsilon}}
\newcommand{\registered}
   {{\scriptsize \ooalign{\hfil\raise0.07ex\hbox{\scriptsize \sc r}\hfil%
              \crcr\mathhexbox20D}}}

\newtheorem{theorem}{Theorem}

\newtheorem{lemma}{Lemma}

\newcommand{\nc}{\newcommand}
\nc{\ds}{\displaystyle}
\nc{\mbZ}{\mathbb Z}
\nc{\mbQ}{\mathbb Q}
\nc{\mbR}{\mathbb R}
\nc{\mbC}{\mathbb C}
\nc{\mbN}{\mathbb N}
\nc{\mbE}{\mathbb E}
\nc{\mbP}{\mathbb P}

\nc{\PH}{\emph{PH} }
\nc{\ME}{\emph{ME} }
\nc{\LST}{\emph{LST} }
\nc{\rank}{\mbox{rank\hspace{1pt}}}

\numberwithin{equation}{section}

\numberwithin{figure}{section}

\usepackage[english]{babel}

\begin{document}

 \title{Parameter estimation for discretely-observed linear birth-and-death processes}
  \author{A.~C.~Davison\hspace{.2cm}\\
   \small{ Institute of Mathematics, Ecole Polytechnique F\'ed\'erale de Lausanne} \\~\\
    S.~Hautphenne \\
    \small{School of Mathematics and Statistics, The University of Melbourne;}\\
   \small{Institute  of Mathematics, Ecole Polytechnique F\'ed\'erale de Lausanne}\\ \\ A.~Kraus\\\small{Masaryk University} }

\date{}

  \maketitle

\begin{abstract}
Birth-and-death processes are widely used to model the development of biological populations. Although they are relatively simple models, their parameters can be challenging to estimate, because the likelihood can become numerically unstable when data arise from the most common sampling schemes, such as annual population censuses.  Simple estimators may be based on an embedded Galton-Watson process, but this presupposes that the observation times are equi-spaced.  We estimate the birth, death, and growth rates of a linear birth-and-death process whose population size is periodically observed via an embedded Galton--Watson process, and by  maximizing a saddlepoint approximation to the likelihood. We show that a Gaussian approximation to the saddlepoint-based likelihood connects the two approaches, we establish consistency and asymptotic normality of quasi-likelihood estimators, compare our estimators on some numerical examples, and apply our results to census data for two endangered bird populations and the H1N1 influenza pandemic.

\end{abstract}

\section{Introduction}
\label{sec:intro}

 Linear birth-and-death processes are among the simplest and most widely-used continuous-time Markovian population models, with many applications in biology, genetics, and ecology \cite{karlin1975first,novozhilov2006biological}. Most estimation methods rely on continuous observation of the population, which allows maximum likelihood estimation of the birth and death rates $\lambda$ and $\mu$ \cite{keiding75}. Continuous measurement is rare in practice, however, and often observations are available only at discrete time points. For example, data in population biology typically arise from annual animal censuses. Likelihood-based estimation is then difficult owing to the cumbersome and numerically unstable form of the probability mass function of the population size.

Several authors have proposed approaches to address this instability, and more generally to estimate the parameters of discretely-observed \emph{general} birth-and-death processes, for which an analytic expression for the population size distribution may be unavailable. Chen and Hyrien \cite{chen2011quasi} propose quasi- and pseudo-likelihood estimators, but their empirical analysis suggests that these are inferior to the maximum likelihood estimator (MLE). Crawford and Suchard \cite{crawford2012transition} and Crawford \textit{et al.}~\cite{crawford2014estimation} use continued fractions to obtain expressions for the Laplace transform of the population size distribution.  They approach the computation of the likelihood as a missing data problem, and maximize it using an expectation-maximization (EM) algorithm. This is also used in Xu \textit{et al.}~\cite{xu2015likelihood}, who propose a multitype branching process approximation to birth-death-shift processes, but their method has numerical limitations in settings with very large populations. Ross \textit{et al.}~\cite{ross2006parameter} estimate discretely-observed density-dependent Markov population processes using an approximation based on a deterministic model combined with a Gaussian diffusion.

After sketching the properties of linear birth-and-death processes in Section~2,  we propose two alternative efficient ways to estimate their parameters from discrete-time data. The first approach, developed in Section~3, assumes equal inter-observation times, and uses the Galton--Watson process embedded at the successive observation times. We derive asymptotic properties of the resulting estimators for $\lambda$, $\mu$, and for the growth rate $\omega=\lambda-\mu$.  Keiding \cite{keiding75} used a similar approach, but under the additional assumption that one observes the number of individuals with no offspring at each generation in the embedded process. His estimators correspond to MLEs, while in our case only the estimator of $\omega$ is the MLE from the discretely-observed process.

Our second approach, detailed in Section 4, needs no constraint on the inter-observation times, and is based on  a saddlepoint approximation to the population size distribution, leading to saddlepoint MLEs. Saddlepoint methods provide accurate approximations in many applications: Daniels \cite{daniels1982saddlepoint} applies them to general birth processes; A{\i} and Yu \cite{ai2006saddlepoint} use them to construct closed-form
approximations to the transition densities of continuous-time Markov processes; 
and Pedeli \textit{et al.}~\cite{pedeli2015likelihood} use them to estimate high-order integer-valued autoregressive processes, and show the high accuracy of the resulting approximate MLEs. We highlight the gain in computational efficiency when using the saddlepoint technique to approximate the population size distribution compared to evaluating its analytical form, discuss the approximation error, show that the estimators obtained with the embedded Galton--Watson approach correspond to MLEs based on a Gaussian simplification of the saddlepoint approximation, and establish consistency and asymptotic normality of these estimators for unequal observation intervals. 

Simulations in Section 5 indicate that the relative error between the saddlepoint and true MLEs quickly becomes negligible as the observed population size increases. 
The saddlepoint MLEs always have lower mean square error than the Galton--Watson estimators, but they become comparable if one increases the number of discrete observations of a non-extinct population, or its size at the first observation time.  Section~6 illustrates our ideas with real censuses of two endangered bird populations and the H1N1 influenza pandemic.
All lengthy proofs may be found in Appendix A. 

\section{Discretely observed linear birth-and-death processes}


A linear birth-and-death process (LBDP) with birth rate $\lambda$ and death rate $\mu$ is a continuous-time Markovian branching process in which an individual lifetime is exponentially distributed with parameter $\mu$ and reproduction occurs according to a Poisson process with parameter $\lambda$. The population size increases from $k$ to $k+1$ at rate $k\lambda$, or decreases to $k-1$ at rate $k\mu$. The population growth rate, $\omega:=\lambda-\mu$, is also called the \emph{Malthusian parameter}. 

The conditional probability mass function of the population size $Z(t)$ at time $t$, given $Z(0)=a$, is denoted by ${p}_k(t;a):=\mathsf{P}\{Z(t)=k\,|\,Z(0)=a\}$ for $a,k\in \mathbb{Z}^+.$ The corresponding probability generating function (PGF) is
$f(s,t;a):=\sum_{k\geq 0} p_k(t;a) s^k=f(s,t)^a,$ where $f(s,t):=f(s,t;1)$
satisfies the backward Kolmogorov equation
$$
\frac{\partial f(s,t)}{\partial t}= \mu -(\lambda+\mu) f(s,t) +\lambda^2 f(s,t)^2,\qquad f(s,0)=s,
$$
and this has explicit solution
\begin{equation}\label{fst}
f(s,t)=\left\{\begin{array}{ll}1+\dfrac{(\lambda-\mu)(s-1)}{(\lambda s-\mu)\,\exp{(\mu-\lambda)t}-\lambda(s-1)},& \lambda\neq\mu,\\[1em]1+\dfrac{(s-1)}{1-\lambda t (s-1)},& \lambda=\mu.\end{array}\right.
\end{equation} 
The distribution of $Z(t)$ conditional on $Z(0)=1$ corresponds to a modified geometric distribution,
$p_k(t;1)=\alpha(t) \,\mathds{1}_{\{k=0\}}+\{1-\alpha(t)\} \{1-\beta(t)\} \beta(t)^{k-1}\,\mathds{1}_{\{k>0\}}$, where  $\mathds{1}_{\{\cdot\}}$ denotes the indicator function, with 
$$
\alpha(t)=\left\{\begin{array}{ll} \dfrac{\mu\{\exp{\omega t}-1\}}{\lambda \exp{\omega t}-\mu},& \lambda\neq\mu,\\[1em]\dfrac{\lambda t}{1+\lambda t},&\lambda=\mu,\end{array}\right.\qquad
\beta(t)=\left\{\begin{array}{ll} \dfrac{\lambda}{\mu}\alpha(t), & \lambda\neq\mu,\\[1em]
\alpha(t),& \lambda=\mu.\end{array}\right.
$$
Conditionally on $Z(0)=a$ and by independence of the individuals, $p_0(t;a)=\alpha(t)^a$, and for $k\geq 1$, by conditioning on the number of ancestors whose descendence becomes extinct by time $t$, we obtain \cite[p.~47, with some typos corrected here]{guttorp91}:
\begin{equation}\label{pmfpos}
p_k(t;a)=\sum_{j=\max(0,a-k)}^{a-1}  \binom {a}{j}  \binom {k-1}{a-j-1} \alpha(t)^j \,[\{1-\alpha(t)\} \{1-\beta(t)\}]^{a-j} \,\beta(t)^{k-a+j}.
\end{equation}
The conditional mean and variance of the population size at time $t$ are respectively
\begin{align}\label{mta}
m(t;a)&=\mathsf{E} \{Z( t) \, | \, Z(0)=a \} =a\,\exp(\omega t),\\ 
\label{sigta}
\sigma^2(t;a)&= \mathsf{Var} \{Z(t) \, | \, Z(0)=a \}  =\left\{\begin{array}{ll}a\, \frac{\lambda + \mu}{\omega}\, \exp(\omega t) \{\exp(\omega t) - 1\},& \lambda\neq\mu,\\[1em]a \,2\lambda t,& \lambda=\mu.\end{array}\right.
\end{align}
When $a=1$, we write $m(t):=m(t;1)$ and $\sigma^2(t):=\sigma^2(t;1)$.

Suppose that $M$ independent trajectories of an LBDP are discretely observed $N+1$ times, with $M,N\geq 1$, and for $i=1,\ldots, M$ and $j=0,\ldots, N$ let $t_{i,j}$ and $k_{i,j}$ denote the time and population size corresponding to the $(j+1)$st observation of the $i$th trajectory.  The log-likelihood based on observations $\vc t:=\{t_{i,j}\}$ and $\vc k:=\{k_{i,j}\}$ is  
\begin{eqnarray}
\ell(\lambda,\mu; \vc t,\vc k)\label{eq1}&= &\sum_{i=1}^M\log \mathsf{P}\{Z(t_{i,1})=k_{i,1},\ldots,Z(t_{i,N})=k_{i,N}\,|\,Z(0)=k_{i,0}\}\\
\label{eq2}&=&\sum_{i=1}^M\sum_{j=1}^N\log \mathsf{P}\{Z(t_{i,j}-t_{i,j-1})=k_{i,j}\,|\,Z(0)=k_{i,j-1}\},
 \end{eqnarray}
where \eqref{eq2} follows from the Markov property. Computing the MLEs of $\lambda$ and $\mu$ therefore requires the efficient calculation of ${p}_k(t;a)$ for any $a,k\in \mathbb{Z}^+$ and $t\in \mathbb{R}^+$. Despite the explicit form~\eqref{pmfpos} of these probabilities, the binomial coefficients for large values of $k$ or $a$ are very large, leading to numerical difficulties and making the maximization of \eqref{eq2} cumbersome, if not impossible. These issues, mentioned earlier in the literature (for example in \cite{darwin1956behaviour,chen2011quasi}), hamper the use of likelihood estimation of  $\lambda$ and $\mu$ from discrete observations. In the next two sections we propose methods to circumvent them. 

\section{Embedded Galton--Watson process approach}\label{sec.gwp}

If the inter-observation times all equal $\Delta t$, the discretely-observed LBDP corresponds to a Galton--Watson (GW) process with progeny generating function $P(s):=f(s,\Delta t)$  and offspring mean and variance $m:= m(\Delta t;1)$ and $\sigma^2:=\sigma^2(\Delta t;1)$, where $m(t;a)$ and $\sigma^2(t;a)$ are defined in~\eqref{mta} and~\eqref{sigta} \cite[p.~101]{harris2002theory}.
The birth and death rates may be expressed as
\begin{equation}
\label{lambda_m&sig}
\lambda = \frac{\log m}{2\Delta t} \left\{\frac{\sigma^2}{m(m-1)} +1 \right\},\quad \mu = \frac{\log m}{2\Delta t} \left\{ \frac{\sigma^2}{m(m-1)} - 1 \right\}, 
\end{equation} in the non-critical case, and 
\begin{equation}\label{lambda_mcrit} \lambda=\mu=\frac{\sigma^2}{2t}
\end{equation}in the critical case,
which allows simple estimation of them using the offspring mean and variance of~the~embedded GW process. We call this the \textit{GW approach}. We start by sketching the properties of these estimators based on a single trajectory, $M=1$. 

Based on the observation of the initial population size $Z_0:=Z(0)$ and the successive population sizes $Z_1:=Z(\Delta t), \ldots,Z_N:=Z(N\Delta t)$ in the $N$ subsequent generations of the embedded GW process,
the natural estimators of $m$ and $\sigma^2$ are 
\begin{equation}
\label{m_GW}
\hat{m}_{Z_0,N}  =  \frac{\sum_{n=1}^{N} Z_n}{\sum_{n=0}^{N-1} Z_n},\quad 
\widehat{\sigma^2}_{Z_0,N}  =  \frac{1}{N}\, \sum_{n=1}^{N} Z_{n-1} \left( \frac{Z_{n}}{Z_{n-1}} - \hat{m}_{Z_0,N}\right)^2.
\end{equation}
In particular, $\hat{m}_{Z_0,N}$ corresponds to a method of moments, conditional least squares and MLE of $m$  \cite[Chapter~2]{guttorp91}. There is a possibility of division by zero in $\widehat{\sigma^2}_{Z_0,N}$,  but if $Z_{n-1}=0$, then $Z_{n}=0$ as well, and we may set $Z_{n}/Z_{n-1}:=1$ so that the corresponding term vanishes. This only happens if the process becomes extinct.

Asymptotic properties of $\hat{m}_{Z_0,N}$ and $\widehat{\sigma^2}_{Z_0,N}$ have been studied extensively.  As $N\to\infty$, and conditionally on the eventual explosion of the process, $\{Z_N\to\infty\}$, both estimators are consistent and asymptotically normal \cite{guttorp91}:
\begin{eqnarray}
\label{m_gw_dist}
\sqrt{\frac{\sum_{n=1}^N Z_{n-1}}{\sigma^2}} \big( \hat m_{Z_0, N} - m \big) & \xrightarrow[]{d} & \mathcal{N}(0, 1), \\
\label{sig_gw_dist}
\sqrt{\frac{N}{2\sigma^4}} \big( \widehat{\sigma^2}_{Z_0, N} - \sigma^2 \big) & \xrightarrow[]{d} & \mathcal{N}(0, 1),
\end{eqnarray}
 where $\xrightarrow[]{d}$ denotes convergence in distribution. These results hold for any fixed value of $Z_0$ but only for supercritical processes ($m>1$), since otherwise eventual explosion has probability zero. 

If both $Z_0\to\infty$ and $N\to\infty$, and conditionally on $\{Z_N\to\infty\}$, it can be shown that the estimators are consistent and jointly asymptotically normal \cite{duby&rouault82}:
\begin{equation}
\label{m&sig_gw_dist}
\left(
\begin{matrix}
\sqrt{\frac{Z_0\,(m^N-1)}{\sigma^2\,(m-1)}}\, (\hat m_{Z_0, N} - m ) \\
\sqrt{\frac{N}{2\sigma^4}}(\widehat{\sigma^2}_{Z_0, N}-\sigma^2)\\
\end{matrix}
\right)
\xrightarrow[]{d} \mathcal{N}
\left(
\mathbf{0}, \mathbf{I} 
\right),
\end{equation}
where $\mathcal{N}(\mathbf{0}, \mathbf{I} )$ is the bivariate normal distribution with independent standard normal marginals.

If $Z_0\to\infty$ but $N$ is fixed, results on consistency and asymptotic normality are available for $\hat m_{Z_0, N}$ but not for $\widehat{\sigma^2}_{Z_0, N}$. For critical and subcritical processes, for which $m=1$ and $m<1$ respectively, results on consistency and marginal asymptotic normality of $\hat m_{Z_0, N}$ and $\widehat{\sigma^2}_{Z_0, N}$ are also available if both $Z_0\to\infty$ and $N\to\infty$, under additional assumptions on their relative speed of convergence \cite{guttorp91,dion&yanev94}. We shall not explore these cases further here.

\label{GW_M1}

On inserting the estimators $\hat m_{Z_0, N}$ and $\widehat{\sigma^2}_{Z_0,N}$ from~\eqref{m_GW} into~\eqref{lambda_m&sig}, we obtain estimators 
\begin{eqnarray}
\label{lambda_GW}
\hat\lambda_{Z_0, N} & = & \frac{\log \hat m_{Z_0, N}}{2\Delta t} \left\{\frac{\widehat{\sigma^2}_{Z_0, N}}{\hat m_{Z_0, N}(\hat m_{Z_0, N}-1)} +1 \right\},
\\
\label{mu_GW}
\hat\mu_{Z_0, N} & = & \frac{\log \hat m_{Z_0, N}}{2\Delta t} \left\{ \frac{\widehat{\sigma^2}_{Z_0, N}}{\hat m_{Z_0, N}(\hat m_{Z_0, N}-1)} - 1 \right\},
\end{eqnarray}
whose properties can be derived from those of $\hat m_{Z_0, N}$ and $\widehat{\sigma^2}_{Z_0,N}$, as shown in the next theorem, whose proof is in Appendix A.
\begin{theorem}
\label{lambda&mu_gw_dist}
Suppose that $m>1$. 
Then 
$ \hat\lambda_{Z_0, N} $ and $ \hat\mu_{Z_0, N} $ are consistent on $\{ Z_N\to\infty\}$ as $N\to\infty$, and also as both $N\to\infty$ and $Z_0\to\infty$. 

Moreover,
\[
\sqrt{N}\, \left( 
\begin{matrix}
\hat\lambda_{Z_0, N} - \lambda  \\
\hat\mu_{Z_0, N} - \mu 
\end{matrix}
\right)
\xrightarrow[]{d}
\mathcal{N}
\left\{
\left(
\begin{matrix}
0 \\
0 \\
\end{matrix}
\right)
,
\frac{(\log m)^2 \sigma^4}{2(\Delta t)^2 m^2(m-1)^2}
\times
\left(
\begin{matrix}
1 & 1 \\
1 & 1 \\
\end{matrix}
\right)
\right\}
\]
conditionally on $\{Z_N\to\infty\}$ as $N\to\infty$, and also as both $N\to\infty$ and $Z_0\to\infty$. 
\end{theorem}

By Theorem~\ref{lambda&mu_gw_dist}, the difference $
\sqrt{N}\, \big\{ (\hat \lambda_{Z_0, N} - \lambda) - (\hat \mu_{Z_0, N} - \mu) \big\}$
vanishes in probability conditional on $\{ Z_N\to\infty \}$ as $N\to\infty$ and also as both $N\to\infty$ and $Z_0\to\infty$.
Thus, the estimator $
\hat{\omega}_{Z_0,N} = \hat \lambda_{Z_0, N} - \hat \mu_{Z_0, N}$
of the Malthusian parameter $\omega$ satisfies 
$\sqrt{N}(\hat{\omega}_{Z_0,N}-\omega) \xrightarrow[]{p} 0$
conditionally on $\{Z_N\to\infty\}$ as $N\to\infty$, and also as both $N\to\infty$ and $Z_0\to\infty$, where $\xrightarrow[]{p}$ denotes the convergence in probability.
This is no surprise, as  $\hat{\omega}_{Z_0,N} = \log (\hat m_{Z_0, N})/\Delta t$ is a function of $\hat m_{Z_0, N}$ only. Asymptotic normality for $\hat{\omega}_{Z_0,N}$ is therefore achieved at a higher rate of convergence than for $\hat{\lambda}_{Z_0,N}$ and $\hat{\mu}_{Z_0,N}$.
\begin{theorem}
Suppose that $m>1$. 
Then 
$ \hat\omega_{Z_0, N} $ is consistent on $\{ Z_N\to\infty\}$ as $N\to\infty$ and also as both $N\to\infty$ and $Z_0\to\infty$. 

Moreover,
\[
\sqrt{\sum_{n=1}^N Z_{n-1}}\times \big( \hat \omega_{Z_0, N} - \omega \big) 
\xrightarrow[]{d} 
\mathcal{N}
\left\{
0, \frac{\sigma^2}{m^2 (\Delta t)^2} 
\right\}
\]
conditionally on $\{Z_N\to\infty\}$ as $N\to\infty$, and
\[
\sqrt{Z_0\,(m^N-1)}\times \big( \hat \omega_{Z_0, N} - \omega \big) 
\xrightarrow[]{d} 
\mathcal{N}
\left\{
0,  \frac{\sigma^2 (m-1)}{m^2 (\Delta t)^2} 
\right\}
\]
conditionally on $\{Z_N\to\infty\}$ as $N\to\infty$ and $Z_0\to\infty$.
\end{theorem}
\begin{proof}
Consistency of $\hat\omega_{Z_0, N}$ follows from consistency of $\hat m_{Z_0, N}$;
asymptotic normality of $\hat\omega_{Z_0, N}$ follows from asymptotic normality of $\hat m_{Z_0, N}$ and application of the delta method.
\end{proof}

Suppose now that we observe multiple independent trajectories of the LBDP, i.e., $M>1$, and that the inter-observation times are constant and equal $\Delta t$ for all trajectories, so that we observe $M$ independent trajectories of the embedded GW process.
Let $Z_{i, n}:= Z_i(n\Delta t)$ be the $n^{\text{th}}$ observation of the $i^{\text{th}}$ trajectory for $i\in\{ 1, \ldots, M\}$ and $n\in\{ 0, \ldots, N\}$. The estimators of the offspring mean and variance of the GW process are then
\begin{equation}
\label{m_GW_M}
\hat{m}_{M,N} = \frac{\sum_{i=1}^{M}\sum_{j=1}^{N} Z_{i,j}}{\sum_{i=1}^M\sum_{j=-1}^{N} Z_{i,j-1}}, \quad 
\widehat{\sigma^2}_{M,N} = \frac{1}{MN}\, \sum_{i=1}^{M}\sum_{j=1}^{N} Z_{i,j-1} \left( \frac{Z_{i,j}}{Z_{i,j-1}} - \hat{m}_{M,N}\right)^2.
\end{equation} 
Again, we set $0/0:=1$ if any of the trajectories becomes extinct. These estimators appear not to have been analysed previously. We derive them in Section~\ref{gauss}, and establish the consistency and asymptotic normality of the corresponding estimators for $\lambda$ and $\mu$.

\section{Saddlepoint approximation}

Saddlepoint approximation offers a very accurate way to approximate probability distributions \cite{daniels1954saddlepoint,butler2007saddlepoint}. It is particularly useful when these distributions are intractable, or when their numerical evaluation is unstable, as in our context. 

The conditional {cumulant generating function} (CGF) of the population size $Z(t)$ at time $t$, given $Z(0)=a$, is defined as $K(x,t;a):=\log M(x,t;a),$ where $M(x,t;a):=f(\exp(\omega t),t;a)$ is the conditional moment generating function. Its maximal convergence set in a neighbourhood of $0$ is 
 $\mathcal{S}(t)=(-\infty,\log R(t))$, where 
 $$R(t)=\left\{\begin{array}{ll}1+\dfrac{\omega}{\lambda \{\exp(\omega t)-1\} },& \lambda\neq\mu,\\[1em]1+\dfrac{1}{\lambda t},& \lambda=\mu.\end{array}\right.$$ 
When $k=0$ we have $p_0(t;a)=\alpha(t)^a$ and no approximation is needed. For $k=1,2,\ldots$, the (first-order) saddlepoint approximation of $p_k(t;a)$ is 
\begin{equation}\label{spa_bd} \tilde{p}_k(t;a)=\dfrac{1}{\sqrt{2\pi K''(\tilde{x},t; a)}}\exp\{K(\tilde{x},t; a)-\tilde{x} k\},\end{equation}where $\tilde{x}=\tilde{x}(k,t;a)$, called the saddlepoint, corresponds to the unique solution in $\mathcal{S}(t)$ to the equation
\begin{equation}\label{saddle_equ1}
K'(\tilde{x},t;a):=\frac{\partial}{\partial x}K(x,t;a)\Big|_{x=\tilde{x}}=k.
\end{equation}
Since $K(x,t;a)=a K(x,t;1)$, \eqref{spa_bd} can be rewritten as 
\begin{equation}\label{spa_bd_a} \tilde{p}_k(t;a)=\dfrac{1}{\sqrt{2\pi aK''(\tilde{x},t; 1)}}\exp\{a[K(\tilde{x},t; 1)-\tilde{x} (k/a)]\},\end{equation}where $\tilde{x}=\tilde{x}(k/a,t;1)$ is the unique solution in $\mathcal{S}(t)$ to the equation
\begin{equation}\label{saddle_equ2}
K'(\tilde{x},t;1)=k/a.
\end{equation}
The next lemma shows that analytical expressions can be obtained for the saddlepoint $\tilde{x}(k,t;a)$ and the saddlepoint approximation $\tilde{p}_k(t;a)$ in the LBDP case.

\begin{lemma} \label{explicit}In the LBDP case, $\tilde{x}(k,t;a)$ and $\tilde{p}_k(t;a)$ can be expressed explicitly as $\tilde{x}(k,t;a)=\log\tilde{s}$, where
$\tilde{s}:=\tilde{s}(k,t;a)=(2A)^{-1}[ -B+\sqrt{B^2-4 AC}],$ with, for $\lambda\neq \mu$,
$$ \begin{array}{lclcl}A&:=&A(t)&=&\lambda\{m(t)-1\}\{\lambda-\mu m(t)\},\\
B&:=&B(k,t;a)&=&2\lambda\mu\{1+m(t)^2-m(t)-(a/k)m(t)\}\\&&&&+m(t)(\lambda^2+\mu^2)\{(a/k)-1\},\\
 C&:=&C(t)&=&\mu\{m(t)-1\}\{\mu-\lambda m(t)\},\end{array}$$leading to \begin{eqnarray}\label{spa_bd3} \tilde{p}_k(t;a)&=&\dfrac{1}{\sqrt{2\pi a}}\dfrac{1}{\tilde{s}^k} \left\{\dfrac{\mu-\lambda \tilde{s} +\mu(\tilde{s}-1)\, m(t)}{\mu-\lambda \tilde{s} +\lambda(\tilde{s}-1)\, m(t) }\right\}^a\times\\\nonumber&&\left\{-\frac{\{m(t)-1\} m(t) \tilde{s} (\lambda-\mu)^2 \left\{-\lambda^2 \tilde{s}^2+\lambda m(t) \mu \left(\tilde{s}^2-1\right)+\mu^2\right\}}{\{\lambda [m(t) (\tilde{s}-1)-\tilde{s}]+\mu\}^2 \{\lambda \tilde{s}+\mu [-m(t) \tilde{s}+m(t)-1]\}^2}\right\}^{-1/2}, \end{eqnarray}and for $\lambda=\mu$,
$A:=A(t)=\lambda t-(\lambda t)^2$, $B:=B(k,t;a)=2(\lambda t)^2+(a/k)-1$, $ C:=C(t)=-\lambda t-(\lambda t)^2$,  leading to
\begin{eqnarray}\nonumber \tilde{p}_k(t;a)&=&\dfrac{1}{\sqrt{2\pi a}}\dfrac{1}{\tilde{s}^k}  \left\{ 
\dfrac{\lambda t (1-\tilde{s})+\tilde{s}}{1-\lambda t(\tilde{s}-1)}
\right\}^a\left\{  \frac{\lambda t \tilde{s} \left(-\lambda t \tilde{s}^2+\lambda t+\tilde{s}^2+1\right)}{\{\lambda t (\tilde{s}-1)-1\}^2 (-\lambda t \tilde{s}+\lambda t+\tilde{s})^2}\right\}^{-1/2}. \qquad\end{eqnarray}
\end{lemma}

The approximations $\tilde{p}_k(t;a)$ may not sum to unity for $k=0,1,2,\ldots$, and the {normalised} saddlepoint approximation is given by
\begin{equation}\label{nspa_bd}
\bar{p}_k(t;a)=\left\{\begin{array}{ll} p_0(t;a) & k=0,\\[1em]\{1-p_0(t;a)\}\;\dfrac{\tilde{p}_k(t;a)}{\sum_{j\geq 1} \tilde{p}_j(t;a) } & k\geq 1.
\end{array}\right.
\end{equation} 
However, as mentioned in Daniels \cite{daniels1982saddlepoint}, renormalisation complicates likelihood estimation.

\begin{figure}[t]
\begin{subfigure}{.54\textwidth}
  \centering
  \includegraphics[width=0.92\linewidth]{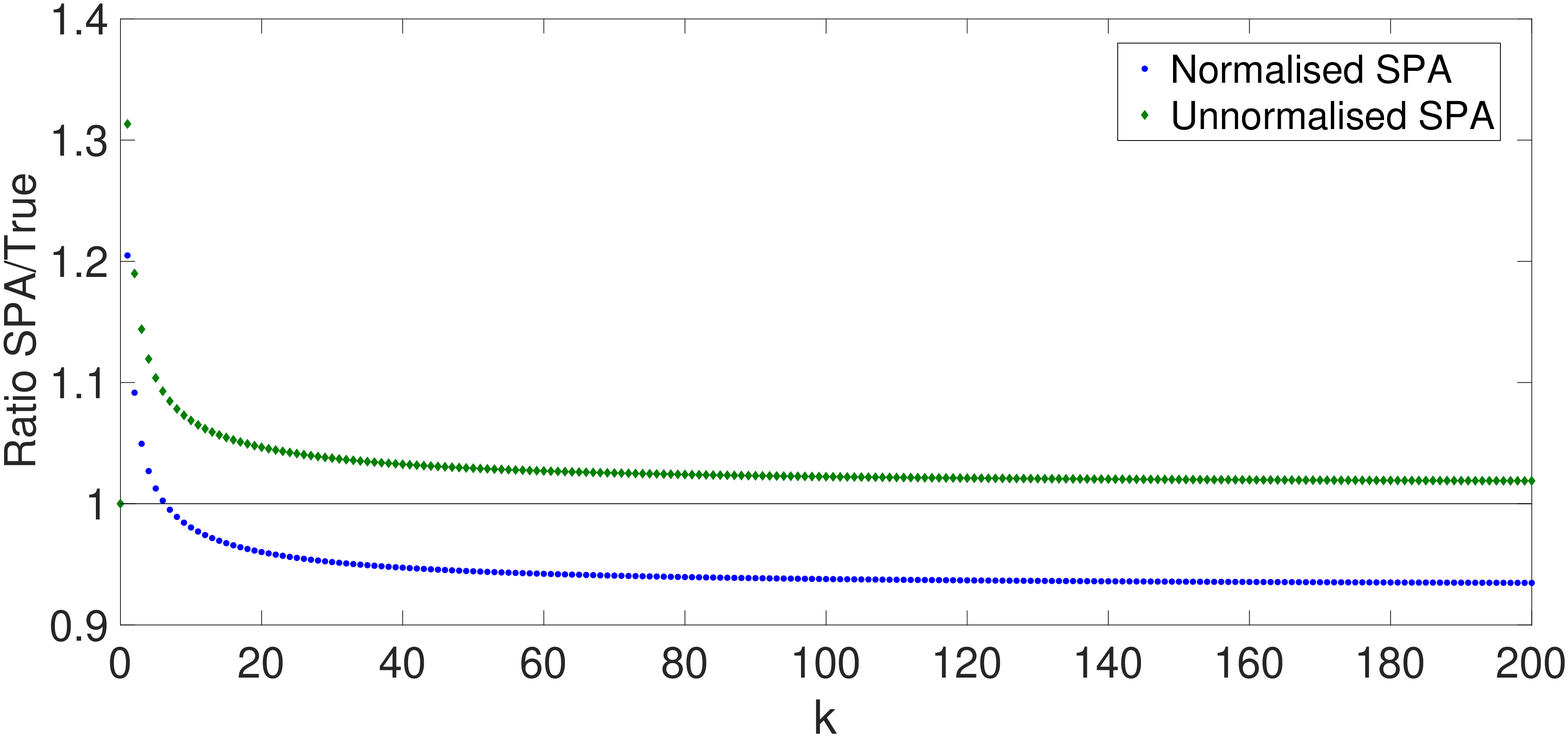}
  \label{fig:sfig1}
\end{subfigure}%
\begin{subfigure}{.54\textwidth}
  \centering
  \includegraphics[width=0.92\linewidth]{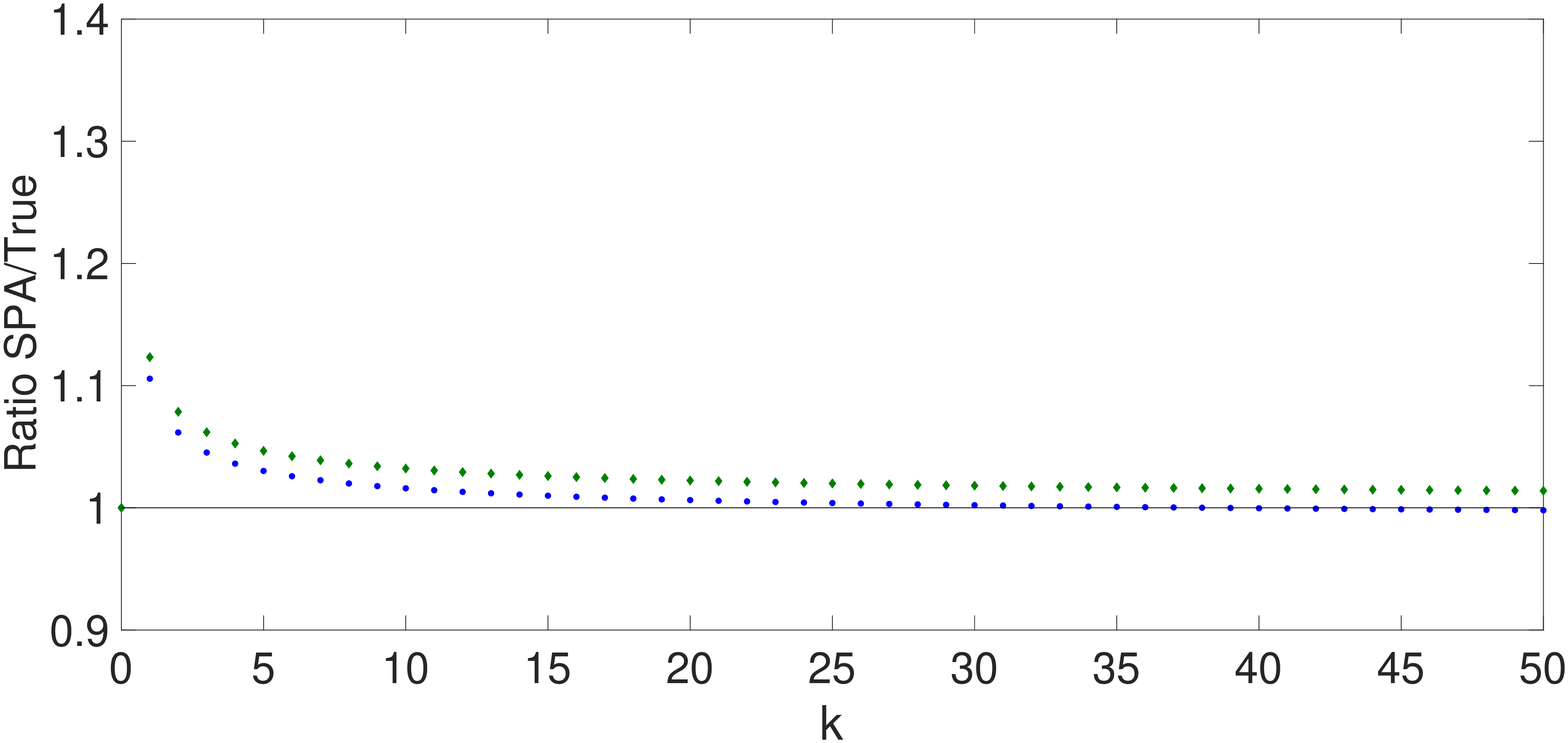}
  \label{fig:sfig2}
\end{subfigure}
\caption{\label{f3b} Ratio between the saddlepoint approximations and the exact probability mass function computed from \eqref{pmfpos} for the population size after $t=1$ in an LBDP with $\lambda=7$ and $\mu=5$ and initial population sizes $a=5$ (left) and $a=20$ (right).} 
\end{figure}

\begin{table}[t]
\centering
\begin{tabular}{ccccccc}
$a$& $5$ & $10$ & $15$ & $20$ & $50$ & $100$\\
\hline
SPA  & 5.9 & 5.7 & 5.9 &6.0 &6.3 & 6.9\\
Exact & 138.9 & 188.6 & 312.3 & 413.7 & 1402.1 & 4939.7
\end{tabular}
\caption{\label{cpu}Average CPU time (ms, based on 100 runs, using Matlab on a 2.6 GHz Intel Core i5 laptop) to compute the saddlepoint approximation (SPA) and the exact probability mass function for population sizes $k=0,1,\ldots,200$ at $t=1$ in an LBDP with $\lambda=7$ and $\mu=5$, as a  function of $Z(0)=a$.}
\end{table}

Figure~\ref{f3b} shows that the ratio between the normalised and unnormalised saddlepoint approximations and the exact probability mass function converges to unity as the initial population sizes $a$ increases: the approximation improves as the numerical computation of the true distribution becomes more challenging. Table~\ref{cpu} compares the CPU times to compute the saddlepoint approximation $\tilde{p}_{k}(1;10)$ and the probability mass function ${p}_{k}(1;10)$ for $k=1$ to 200  and for some parameter values: in contrast with the saddlepoint approximation, the computational burden of the exact distribution increases with $Z(0)$. 

\begin{figure}[t]
\begin{subfigure}{.53\textwidth}
  \centering
  \includegraphics[width=0.92\linewidth]{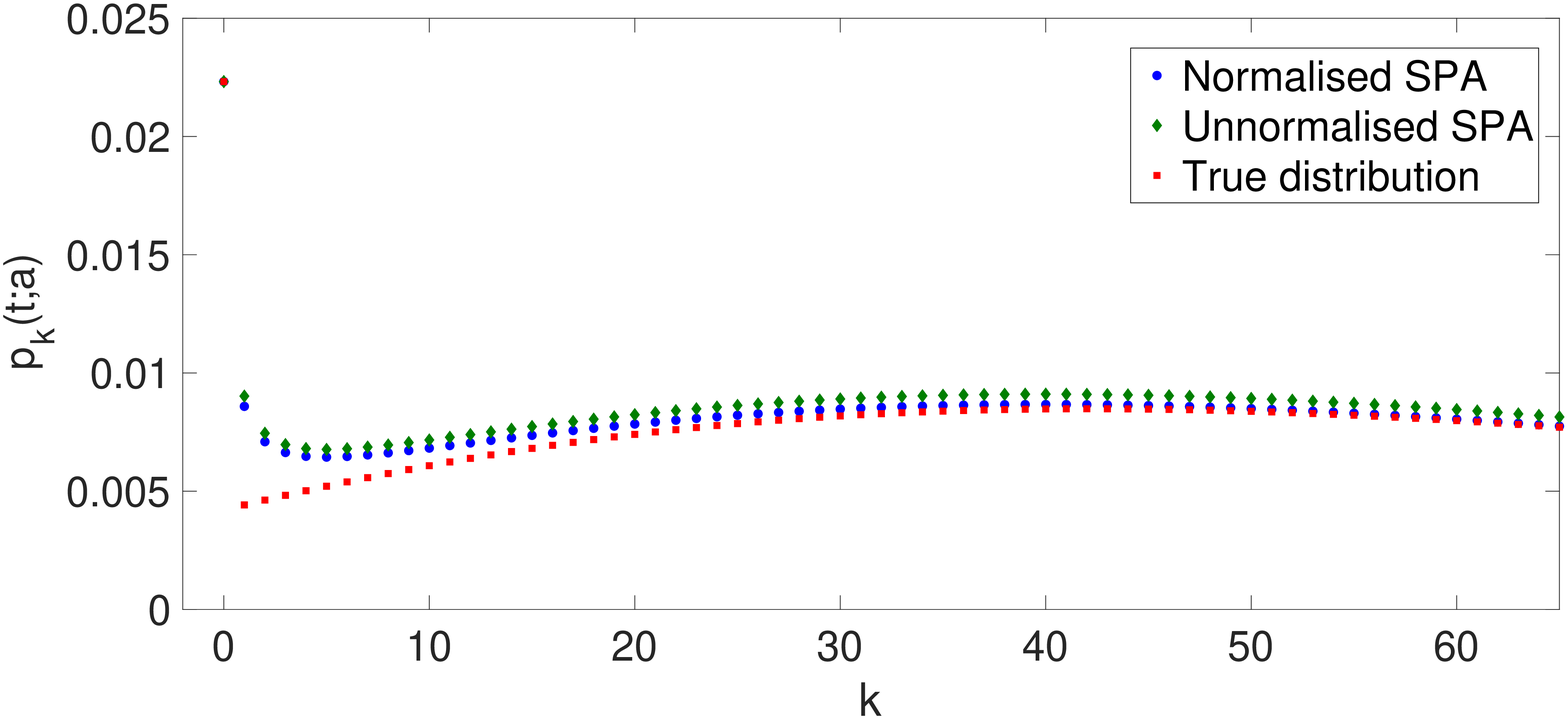}
  \label{fig:sfig1}
\end{subfigure}%
\begin{subfigure}{.53\textwidth}
  \centering
  \includegraphics[width=0.92\linewidth]{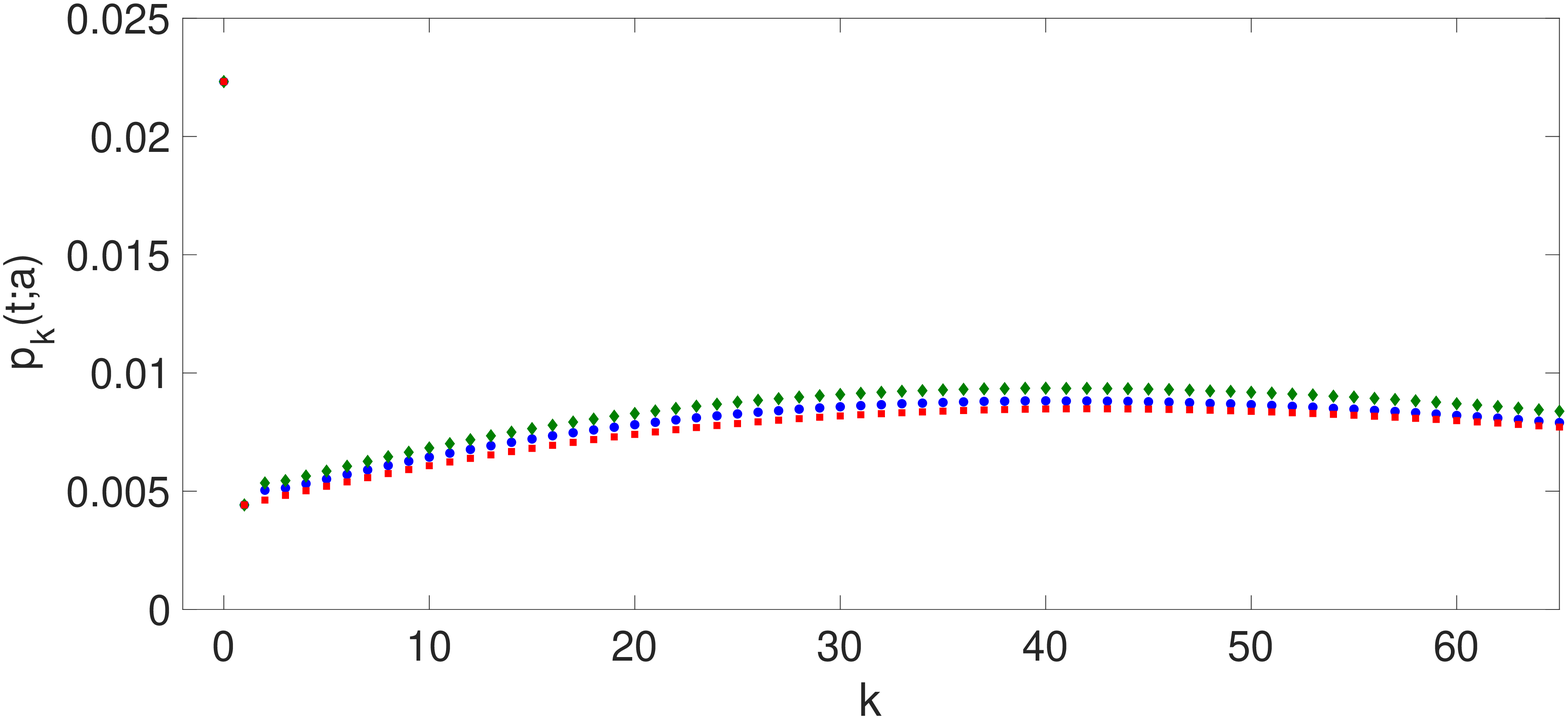}
  \label{fig:sfig2}
\end{subfigure}
\caption{\label{f3}Probability mass function of the population size after $t=1$ in an LBDP with $\lambda=7$ and $\mu=5$ and initial population size $a=10$, together with its saddlepoint approximation $\tilde{p}_{k}(1;10)$ and the normalised saddlepoint approximation $\bar{p}_{k}(1;10)$. Left: unadjusted saddlepoint approximation; Right: conditional saddlepoint approximation.} 
\end{figure}

The left-hand panel of Figure~\ref{f3}
shows that when there is a large gap between $p_0(t;a)$ and $p_1(t;a)$, the saddlepoint approximation has a smoothing effect around $k=0$, which decreases as $|p_0(t;a)-p_1(t;a)|$ decreases, i.e., as the initial population size $a$ increases. The smoothing can be adjusted by applying the saddlepoint technique to the \emph{conditional} CGF, given that the population is not extinct at time $t$, 
$$
K\{x,t;a\,|\,Z(t)>0\}=\log\{M(x,t;a)-p_0(t;a)\}-\log\{1-p_0(t;a)\}, 
$$
though the resulting saddlepoint $\tilde{x}(t,a,k)$ must be obtained numerically. The right-hand panel of Figure~\ref{f3} shows the adjusted saddlepoint approximation. In practice, if  $a$ and/or $k$ are large, this correction is not necessary.

The saddlepoint approximated log-likelihood is 
\begin{equation}\label{spall}\tilde{\ell}(\lambda,\mu; \vc t,\vc k)=\sum_{i=1}^M \sum_{j=1}^N \log \tilde{p}_{k_{i,j}}(\tau_{i,j};k_{i,j-1}),\end{equation}where $\tau_{i,j}:=t_{i,j}-t_{i,j-1}$, and the values $\tilde{{\lambda}}$ and $\tilde{\mu}$ maximising this expression are called the \emph{saddlepoint maximum likelihood estimators} (SPMLEs). We call this the \textit{SPMLE approach}.

As indicated by Figure~\ref{f3b}, the error of the saddlepoint approximation decreases as the initial population size $a$ increases; more precisely (see for example \cite[Chapter 12]{davison2003statistical}),
\begin{equation}
\label{err_spa}{p}_k(t;a)=\tilde{p}_k(t;a)\{1+\mathcal{O}(1/a)\},
\end{equation} 
and this leads to the following lemma.

\begin{lemma}\label{lem_err_1}The error in the saddlepoint approximated log-likelihood is$${\ell}(\lambda,\mu; \vc t,\vc k)-\tilde{\ell}(\lambda,\mu; \vc t,\vc k)=\mathcal{O}\left\{\left(\min_{1\leq i\leq M,1\leq j\leq N}\{k_{i,j-1}\}\right)^{-1}\right\}.$$\end{lemma}

\begin{proof}We have
\begin{eqnarray*}{\ell}(\lambda,\mu; \vc t,\vc k)-\tilde{\ell}(\lambda,\mu; \vc t,\vc k)&=&
\sum_{i=1}^M\sum_{j=1}^N \log\{{p}_{k_{i,j}}(\tau_{i,j};k_{i,j-1})/\tilde{p}_{k_{i,j}}(\tau_{i,j};k_{i,j-1})\}\\&=&\sum_{i=1}^M\sum_{j=1}^N \log\{1+\mathcal{O}(1/k_{i,j-1})\}\quad\mbox{by \eqref{err_spa}}\\&=&\sum_{i=1}^M\sum_{j=1}^N \mathcal{O}(1/k_{i,j-1})\\&=&\mathcal{O}\left\{\left(\min_{1\leq i\leq M,1\leq j\leq N}\{k_{i,j-1}\}\right)^{-1}\right\}.\end{eqnarray*}\end{proof}
Every observed population size in every independent trajectory, except that for the last observation time, plays a role in the order of magnitude of the approximation error. For exploding trajectories where $k_{j,i}\geq k_{j,i-1}$, the error is dominated by the inverse population sizes $k_{j,0}$ at the first observation times, but  for trajectories with a minimum population size close to zero, the error can become non-negligible. In this case,  direct computation of the transition probabilities will be least onerous.  
Alternatively, it may be beneficial to use a multivariate version of the saddlepoint approximation, as described in the Appendix.

\section{Gaussian approximations}\label{gauss}

Second-order Taylor expansion of the CGF $K(x,t;a)$ at $x=0$ yields
\begin{equation}\label{soaK}
K(x,t;a)\approx x K'(0,t;a)+ \frac{x^2}{2} K''(0,t;a)=x m(t;a)+ \frac{x^2}{2} \sigma^2(t;a),
\end{equation} 
and first-order Taylor expansion of $K'(x,t;a)$ at $x=0$ yields
\begin{equation}\label{foaKp}
K'(x,t;a)\approx  K'(0,t;a)+ x K''(0,t;a)= m(t;a)+x \sigma^2(t;a).
\end{equation} 
Using \eqref{saddle_equ1} and \eqref{foaKp}, the saddle point $\tilde{x}(k,t;a)$ can be approximated as
\begin{equation}\label{spaa}
\tilde{x}(k,t;a)\approx \dfrac{k- m(t;a)}{\sigma^2(t;a)},
\end{equation}and using \eqref{soaK} and \eqref{spaa}, the exponent in~\eqref{spa_bd} can be approximated as
\begin{eqnarray}
K(\tilde{x},t; a)-\tilde{x} k &\approx & \tilde{x} \{m(t;a)-k\}+\dfrac{\tilde{x}^2}{2}\sigma^2(t;a)\\&\approx& -\dfrac{\{k-m(t;a)\}^2}{2 \sigma^2(t;a)}.
\end{eqnarray}
Thus the saddlepoint approximation to the conditional probability mass function is roughly
\begin{equation}\label{spaGauss} 
\tilde{\varphi}_k(t;a):=\left\{\dfrac{1}{2\pi \sigma^2(t;a) }\right\}^{1/2} \exp\left[-\dfrac{\{k-m(t;a)\}^2 }{2 \sigma^2(t;a)}\right], 
\end{equation}
i.e.,  the Gaussian density function with mean $m(t;a)$ and variance $\sigma^2(t;a)$. 

This Gaussian approximation also arises from the central limit theorem. If $Z(0)=a$, then $Z(t)\overset{d}{=}\sum_{i=1}^a X_i(t)$, where the i.i.d.\ random variables $X_i(t)$ represent the population size in an LBDP that starts with a single individual at time zero, and have mean and variance $m(t;1)$ and $\sigma^2(t;1)$. 
By the central limit theorem, 
$Z(t)\overset{d}{\approx} \mathcal{N}\{a m(t;1),a\sigma^2(t;1)\}=\mathcal{N}\{m(t;a),\sigma^2(t;a)\}$ for large values of $a$, so $p_k(t;a)\approx \tilde{p}_k(t;a)\approx \tilde{\varphi}_k(t;a)$. However, the saddlepoint approximation $\tilde{p}_k(t;a)$ retains much more information than the Gaussian approximation $\tilde{\varphi}_k(t;a)$, whose only ingredients are the first two moments of the population size.

We now establish a connection between the SPMLE and the GW approaches, and a natural extension of the latter to unequal inter-observation times. The proof of the next lemma can be found in Appendix A.

\begin{lemma} \label{SPA_GW}
If the inter-observation times are equal within and between the trajectories, then the MLEs for $\lambda$ and $\mu$ resulting from~\eqref{spaGauss} coincide with the GW estimators.
\end{lemma}

%
%

If the inter-observation times are unequal, then the approximation~\eqref{spaGauss}, in which $m(t;a)$ and $\sigma^2(t;a)$ are replaced by their expressions \eqref{mta} and \eqref{sigta} in terms of $\lambda$, $\mu$, and the inter-observation times, can still be used to obtain estimators for $\lambda$ and $\mu$. Letting $\omega=\lambda-\mu$ and $\xi=\lambda+\mu$, the logarithm of \eqref{spaGauss} can be rewritten for any $t$ as
\begin{eqnarray}\nonumber
 \log\tilde{\varphi}_k(t;a\,|\,(\omega,\xi))&=&-\frac{1}{2}\log\{2\pi a (\xi/\omega) \exp(\omega t)[\exp(\omega t)-1]\}\\\label{spaGauss_gen}&& -\dfrac{\omega \{k-a \exp(\omega t)\}^2 }{2 a \xi \exp(\omega t)\{\exp(\omega t)-1\}},
\end{eqnarray}
so \eqref{spall} can be approximated by
\begin{eqnarray}\nonumber 
\tilde{\tilde{\ell}}(\omega,\xi; \vc t,\vc k)&=& -\frac{1}{2}\sum_{i=1}^M\sum_{j=1}^N \log\{2\pi k_{i,j-1} (\xi/\omega) \exp(\omega \tau_{i,j})[\exp(\omega \tau_{i,j})-1]\}\\&&\quad -\frac{\omega}{2  \xi}\sum_{i=1}^M \sum_{j=1}^N \dfrac{\{k_{i,j}-k_{i,j-1}\, \exp(\omega \tau_{i,j})\}^2 }{k_{i,j-1} \exp(\omega \tau_{i,j})\{\exp(\omega \tau_{i,j})-1\}},\label{GA_SPA}
\end{eqnarray}
where $\tau_{i,j}=t_{i,j}-t_{i,j-1}$. The resulting MLEs generalise the GW estimators to unequal inter-observation times. The proof of the next theorem is provided in Appendix A.

\begin{theorem}
\label{GASPAthm}
Suppose that $M$ independent replicates of data are available from an LBDP with $\lambda>0$, potentially observed at times $0<t_{i,1}<\cdots<t_{i,N_i}$, where $N_i\geq 1$ and $t_{i,j}-t_{i,j-1}> \epsilon >0$ for $i=1,\ldots, M$ and all $j$.  Suppose that observation for each replicate stops at the earlier of its extinction time $T_i$ and after $N_i$ observations, and that the conditional increments of the process have finite fourth moments almost surely.  Then the estimators of $\lambda$ and $\mu$ resulting from maximisation of~\eqref{GA_SPA} are consistent and asymptotically normally distributed as $M\to\infty$, with covariance matrix given in~\eqref{D.eq}.


%

\end{theorem}

\section{Simulations}

To compare the SPMLE with the  MLE, we simulated independent non-extinct trajectories of LBDPs with $\lambda=7$ and $\mu=5$, with various initial population sizes. For each trajectory, we recorded the population sizes at $N+1=20$ time points and used them to calculate the SPMLE and MLE.  In about a fifth of the trajectories the observed population size exceeded 600 individuals and the  MLE could not be computed. The relative error was never greater than 5\%, and only exceeded 2\% for minimum population sizes smaller than 5. The relative error for $\omega$ is generally smaller than 0.01\%, so the saddlepoint approximation errors in the SPMLEs of $\lambda$ and $\mu$ appear to cancel for $\hat{\lambda}-\hat{\mu}$. 
%
The maximum observed population size has no impact on the CPU time necessary to compute the SPMLE, which was around 0.1s in each case, whereas the CPU time to obtain the exact MLE increases roughly linearly with the population size at rate around 0.2 s/individual.

We ran a series of simulation experiments, either increasing the number $N$ of discrete-time observations, or increasing the number $M$ of independent observed trajectories. 
We first simulated 100 replicates of a single ($M=1$) non-extinct trajectory of an LBDP with $\lambda=7$ and $\mu=6$, starting with $Z(0)=10$ individuals, and we computed the bias, standard deviation and root mean square error (RMSE) of $\hat{\lambda}$ and $\hat{\omega}$ for $N=10, \ldots, 60$. Figure~\ref{fsim1} depicts the results for $\hat{\lambda}$. Both biases decrease rapidly as $N$ increases, with the SPMLE having smaller bias than the GW estimator. The gap between the standard deviation and RMSE of the GW and SPMLE estimator decreases as $N$ increases.

%

\begin{figure}[t]
\begin{subfigure}{.54\textwidth}
  \centering
  \includegraphics[width=0.92\linewidth]{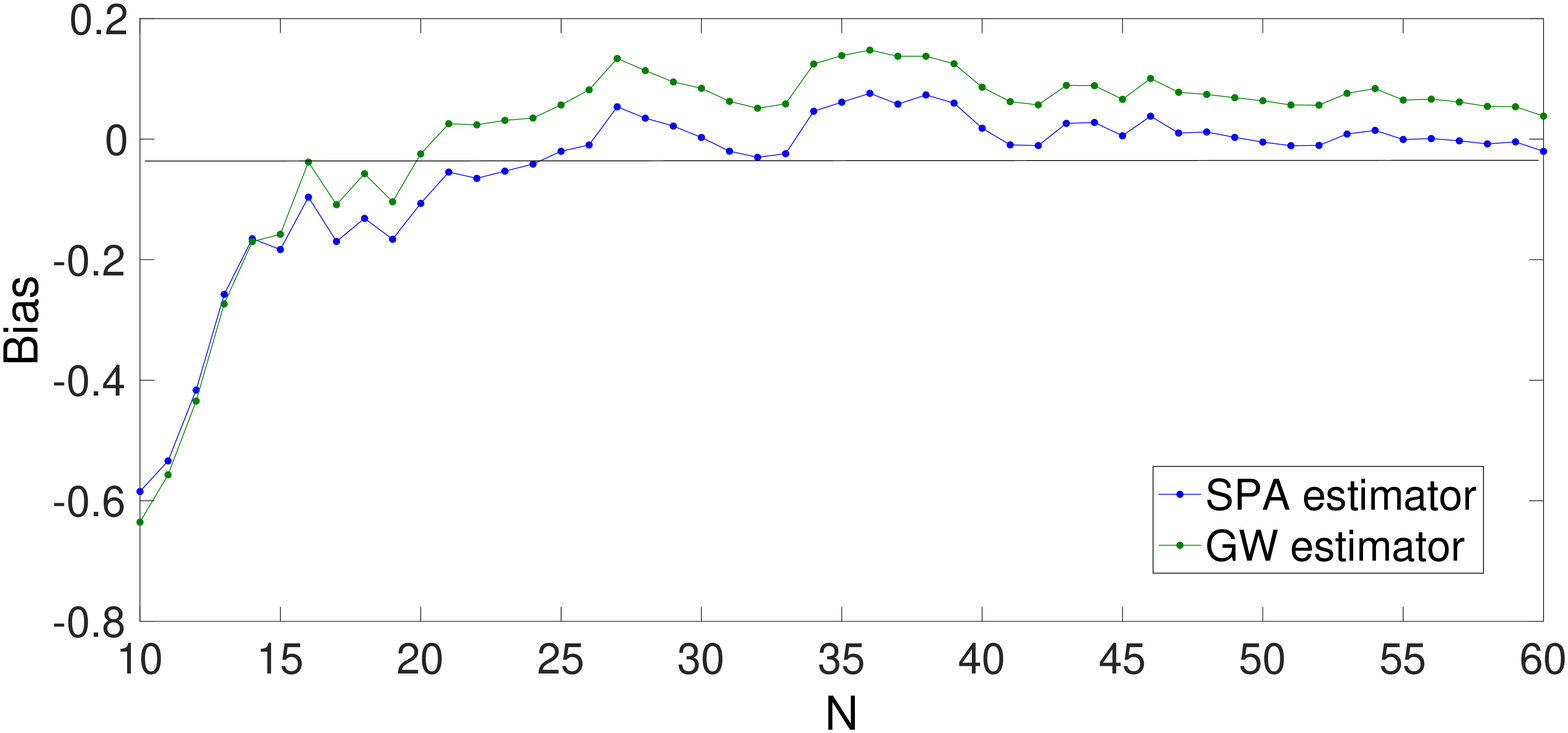}
  \label{fig:sfig1}
\end{subfigure}%
\begin{subfigure}{.54\textwidth}
  \centering
  \includegraphics[width=0.92\linewidth]{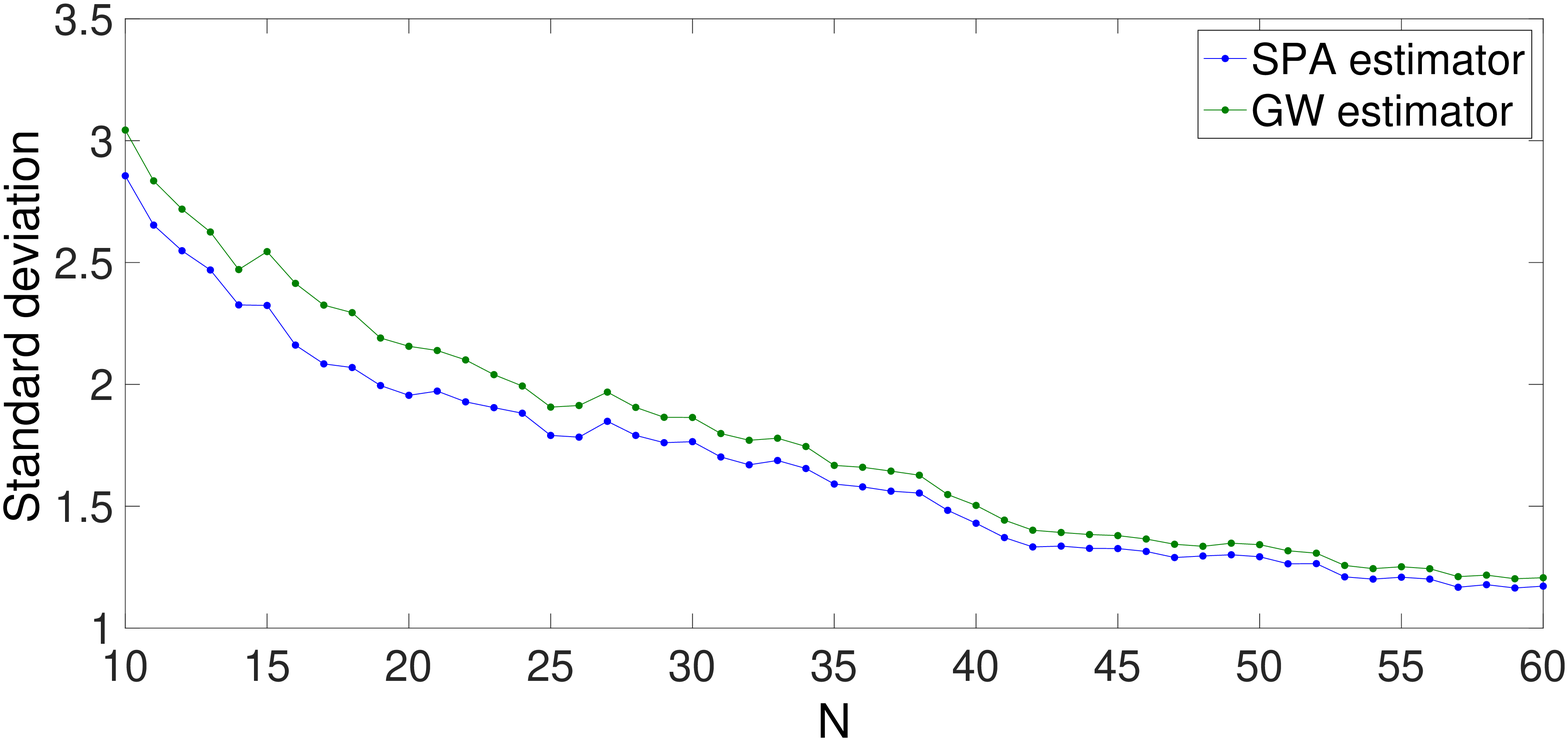}
  \label{fig:sfig2}
\end{subfigure}
\centering
\begin{subfigure}{.54\textwidth}
  \includegraphics[width=0.92\linewidth]{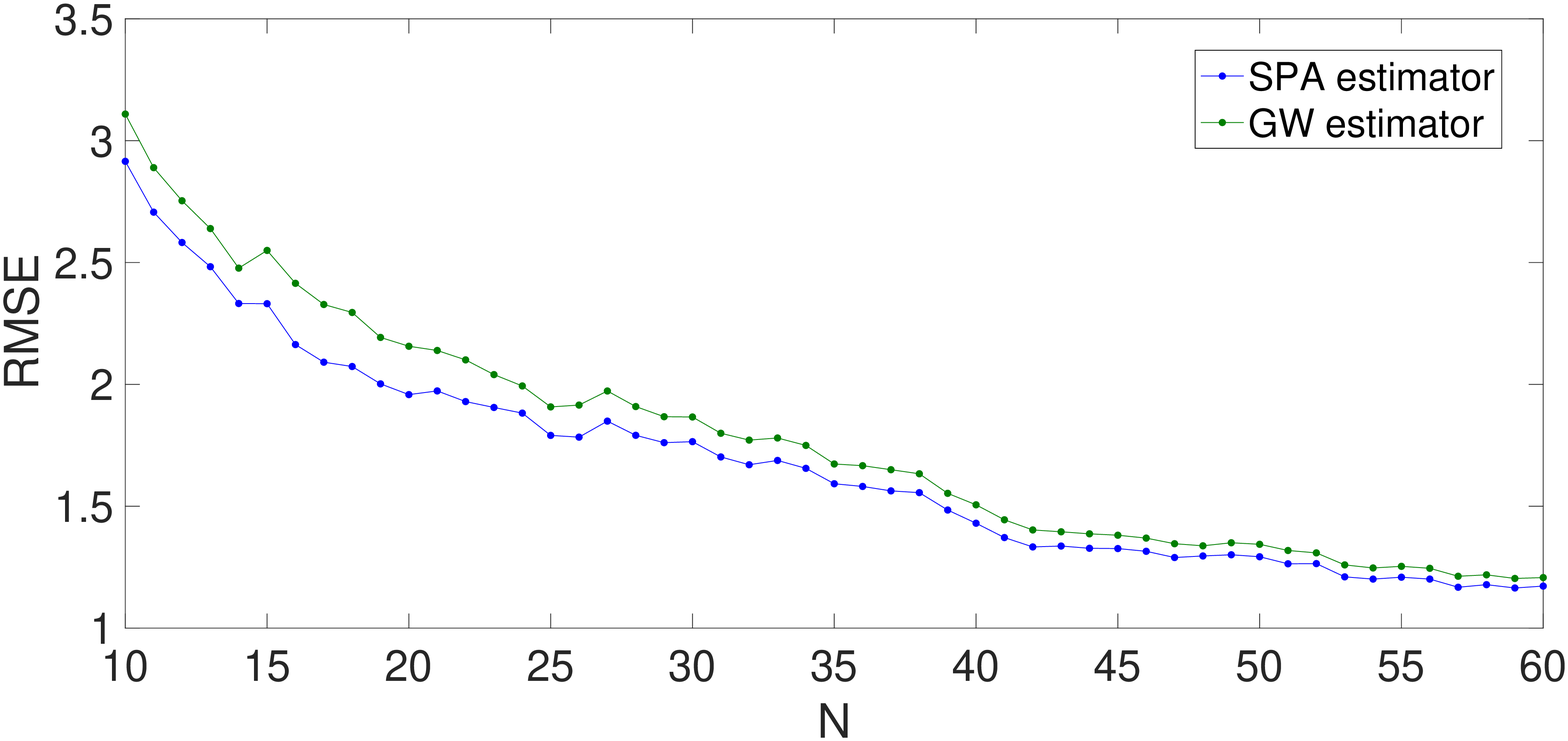}
  \label{fig:sfig2}
\end{subfigure}
\caption{\label{fsim1} Properties of $\hat{\lambda}$ as $N$ increases: Bias, standard deviation and root mean square error for an LBDP with $\lambda=7, \mu=6$, $Z(0)=Z_0=10$. Results based on 100 simulations of a single ($M=1$) non-extinct trajectory observed $N+1$ times, at constant inter-observation times $\Delta t=0.07$. }
\label{fig:fig}
\end{figure}

In a second experiment, we fix $N=30$ and analyse properties of the estimators for $M=1,\ldots,150$, for an LBDP with $\lambda=7$ and $\mu=5$ starting with different initial population sizes $Z(0)=Z_0=1, 10,20$. The bias and RMSE of $\hat{\lambda}$ are presented in Figure~\ref{fsim3}. The quality of the estimators improves rapidly as $M$ increases. The value of $Z_0$ has a clear impact on the quality of the GW estimator, and less impact on that of the SPMLE. We further investigate the effect of $Z_0$ in the RMSE of both estimators for different values of $N$, $M$, and the model parameters. We summarise the results in Table~\ref{tsim1}, which confirms the key role played by $Z_0$ in the quality of the GW estimator, which  becomes comparable to the SPMLE as $Z_0$ increases. For the same value of $M,N$ and $Z_0$, the quality of the estimators increases as the process moves away from criticality ($\omega=0$), as suggested when $m\to 1$  in Theorem \ref{lambda&mu_gw_dist}.

\begin{figure}[t]
\begin{subfigure}{.54\textwidth}
  \centering
  \includegraphics[width=0.92\linewidth]{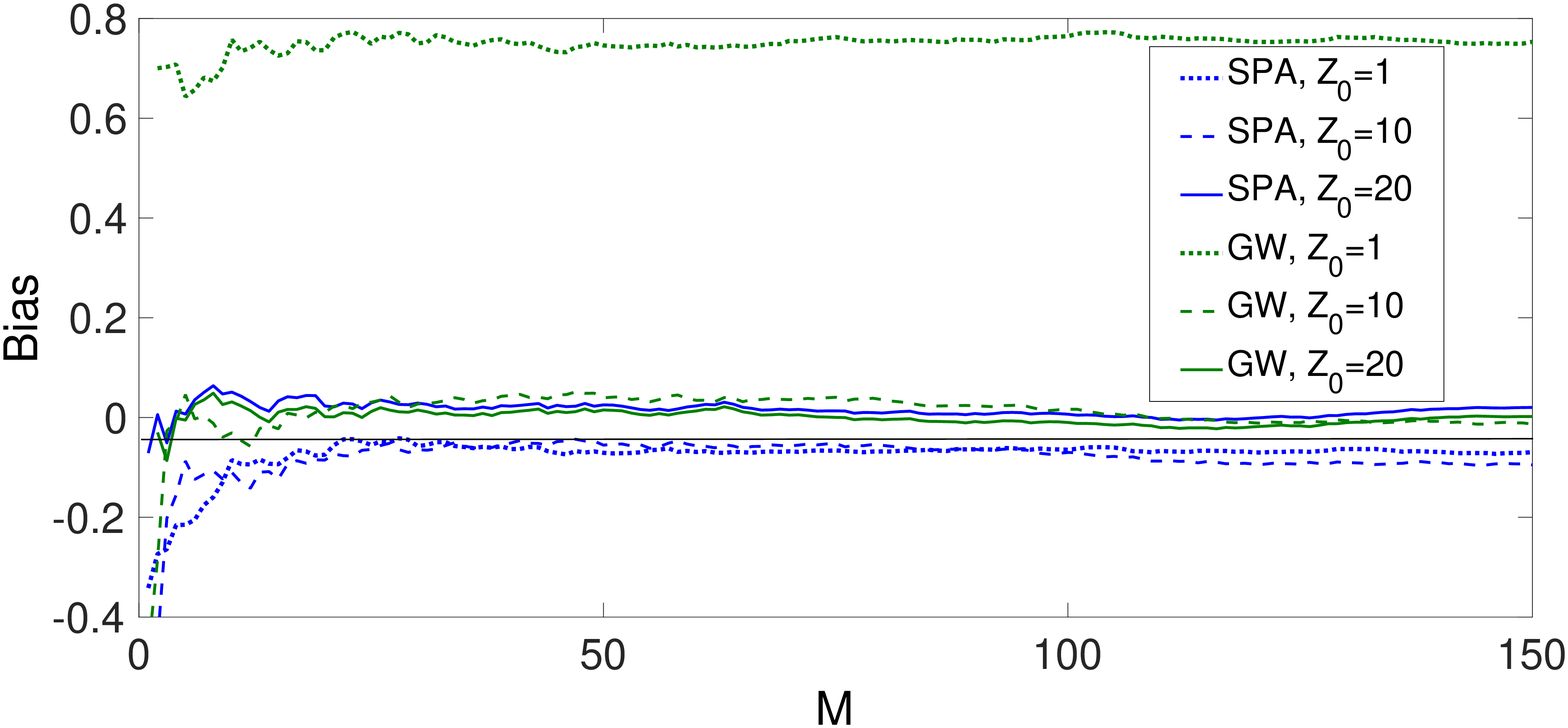}
  \label{fig:sfig1}
\end{subfigure}%
\begin{subfigure}{.54\textwidth}
  \centering
  \includegraphics[width=0.92\linewidth]{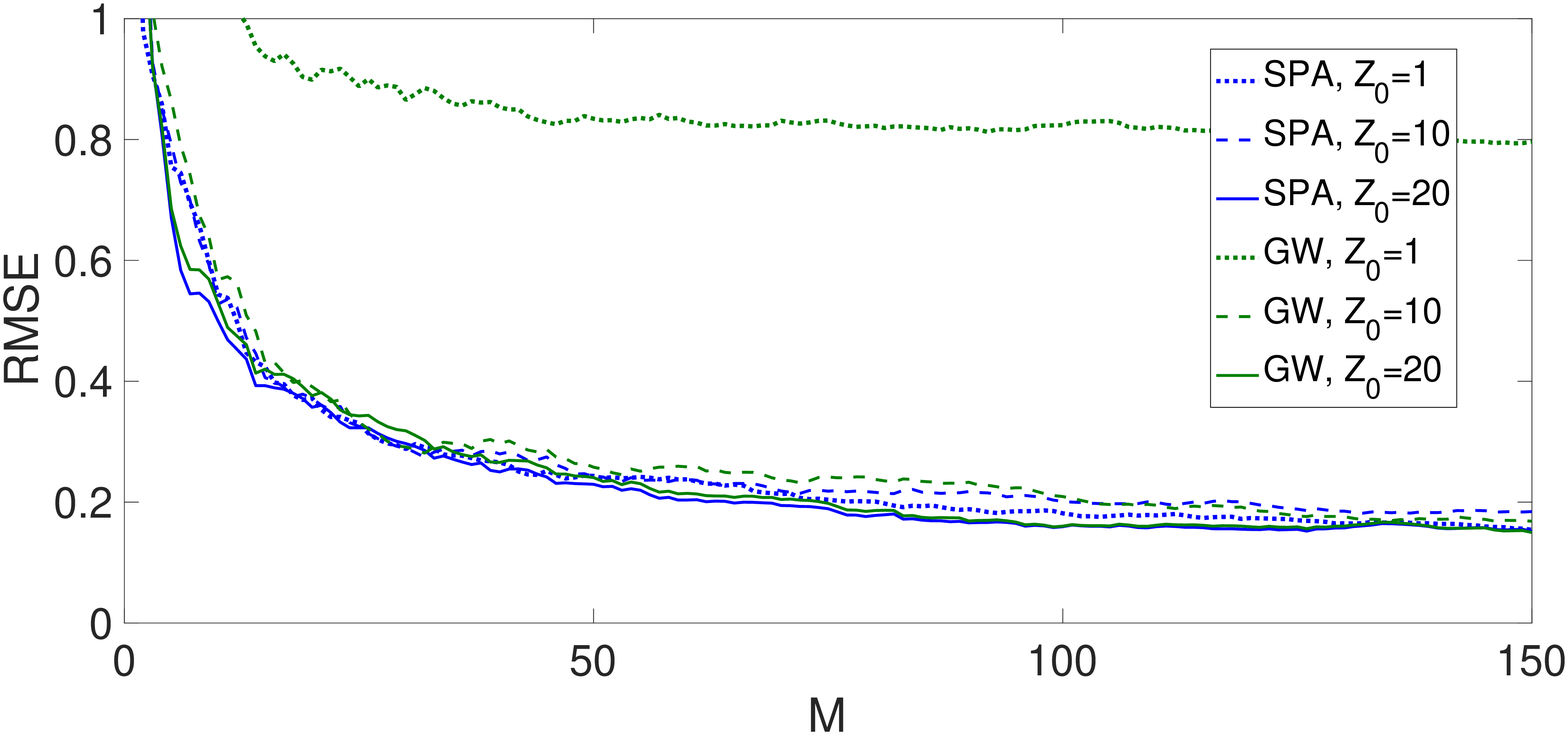}
  \label{fig:sfig2}
\end{subfigure}
\caption{\label{fsim3} Properties of $\hat{\lambda}$ as $M$ increases: Bias and root mean square error for an LBDP with $\lambda=7, \mu=5$, $Z(0)=Z_0=1,10,20$. Results based on 100 simulations of $M$ non-extinct trajectories observed $N+1=30$ times at constant inter-observation times $\Delta t=1/10$. }
\end{figure}

\begin{table}[t]
\small{
\centering
$M=1$,  $\mu=6$, $\omega=1$ 

\medskip
\begin{tabular}{c|ccccc}
& $Z_0=1$ & $Z_0=5$ & $Z_0=10$ & $Z_0=20$ & $Z_0=50$\\ \hline
SPMLE 	&  2.40 & 2.44& 2.40 & 2.45 & 2.48\\
GW & 5.03 & 3.38 & 2.82 & 2.68 & 2.56
\end{tabular}

\bigskip
\smallskip

 \qquad\qquad$M=1$, $\mu=5$, $\omega=2$   \qquad \qquad\qquad\qquad  $M=20$,  $\mu=5$, $\omega=2$
\medskip

\begin{tabular}{c|cccc|cccc}
&  $Z_0=1$ & $Z_0=5$ & $Z_0=10$ & $Z_0=20$ &$Z_0=1$ & $Z_0=5$ & $Z_0=10$ & $Z_0=20$ \\ \hline
SPMLE 	&  1.63 & 1.59& 1.57 & 1.71 &   0.39 & 0.35& 0.35 & 0.36\\
GW & 2.15 & 1.75 & 1.64 & 1.62  &  0.62 &0.40 & 0.36  & 0.36
\end{tabular}}
\normalsize

%
\caption{\label{tsim1}Root mean square error of $\hat{\lambda}$ (true value $\lambda=7$) based on 5000 simulations of $M$ non-extinct trajectories observed $N+1=30$ times at constant inter-observation times $\Delta t=1/10$ for different values of $Z_0$. }
\end{table}

\section{Applications}
We use the GW and MLE approaches to 
estimate the birth and death rates in two bird populations and the transmission and removal rates for an influenza epidemic. In each case a single trajectory ($M=1$) of the population is observed at equally-spaced times.

\begin{figure}
\begin{center}
\includegraphics[width=9cm]{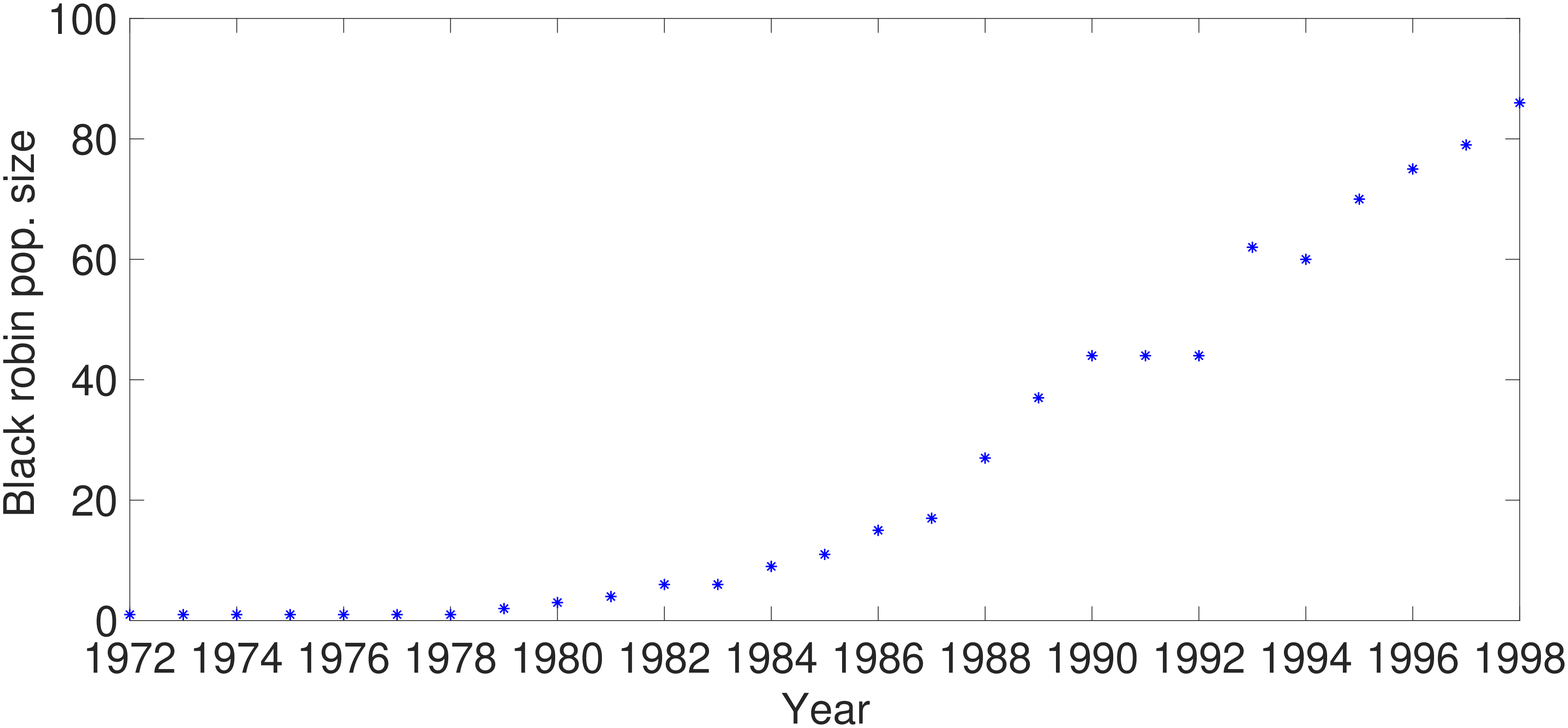}\\
\includegraphics[width=9cm]{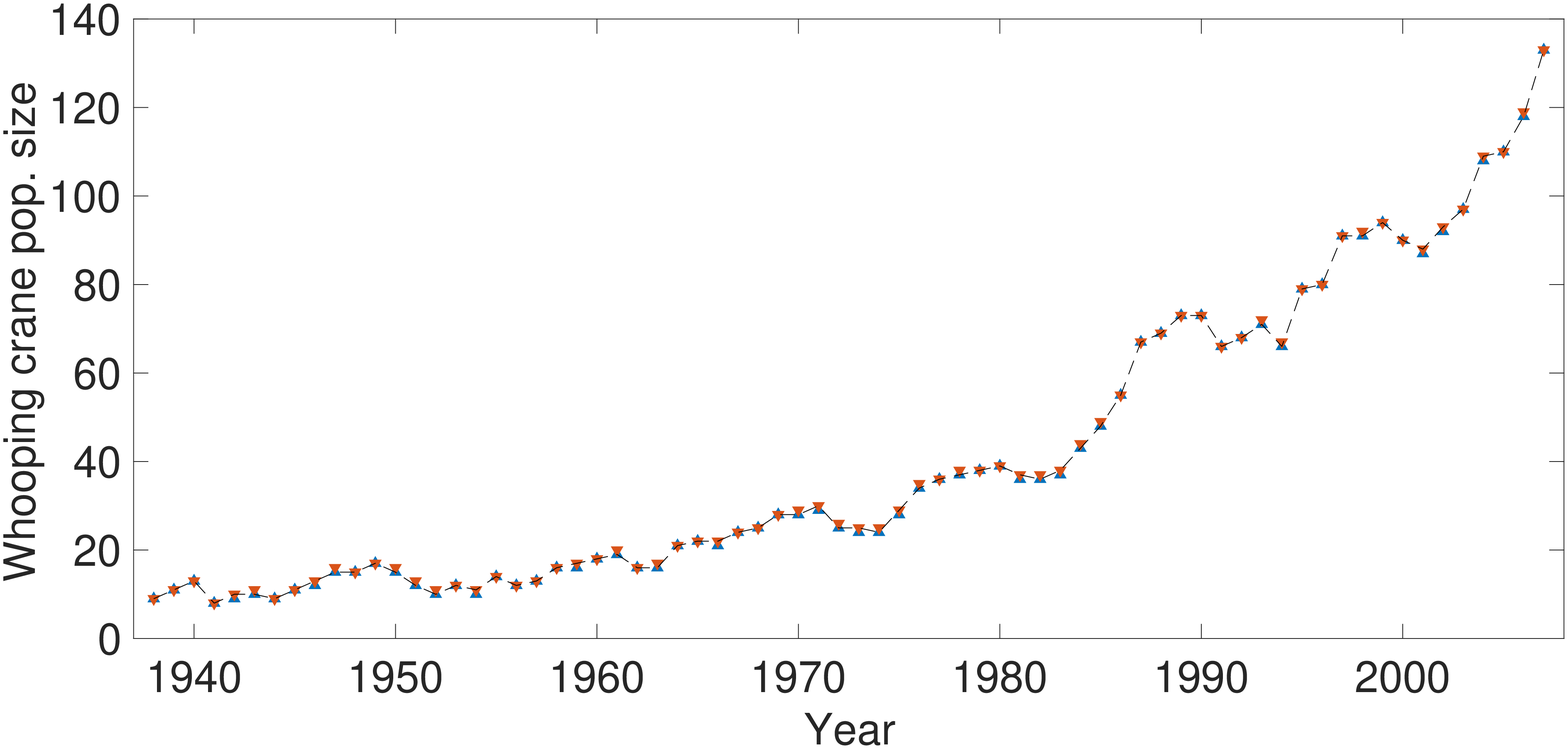}\\
\vspace{2mm}
\includegraphics[width=9cm]{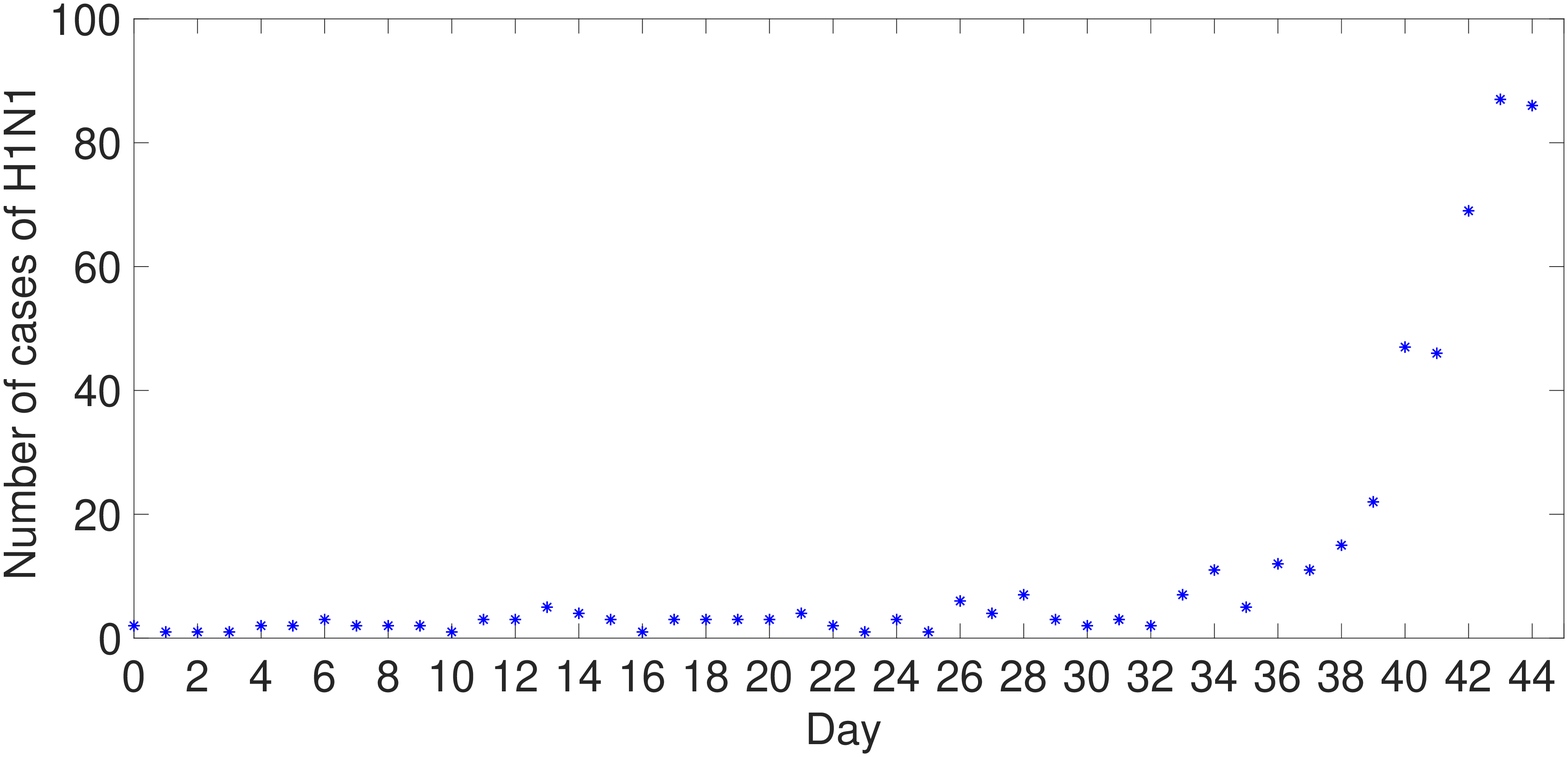}
\end{center}
\vspace{-0.3cm}
\caption{\label{fbr}Data for examples.  \emph{Top}: Yearly black robin population census, 1972--1998, corresponding to females who survived up to at least one year. \emph{Middle:}  Yearly population census of the young and adult whooping cranes arriving in Texas each autumn,  1938--2007. Counts correspond to the number of females, obtained by dividing the total size by two (assuming a 1:1 sex ratio), and are rounded up ({\color{red}$\blacktriangledown$}) or down ($\blacktriangle$) when the value is non-integer. The black line represents a uniform choice between the upper and lower values, which forms our data. \emph{Bottom:} Daily counts ($+1$) of new cases during the H1N1 outbreak in Mexico from 11 March 2009 (Day 0) to 24 April 2009 (Day 44).  }
\end{figure}

\subsection{Black robin population}

The black robin \textit{Petroica traversi} is an endangered songbird species endemic to the Chatham Islands. By 1980, the population had declined to five birds, including only one successful breeding pair \cite{elliston1994black}.  Through intensive conservation efforts in 1980--1989 by the New Zealand Wildlife Service (now the Department of Conservation), the population recovered to 93 birds by spring 1990  \cite{kennedy2014severe,massaro2013human}.  Over the next decade (1990--1998), the population was closely monitored, but without human intervention, and grew to 197 adults by 1998 \cite{kennedy2009}. After this period its growth slowed considerably and it only reached 239 adults in 2011 \cite{massaro2013nest}  and 298 in 2014. 

We fit an LBDP to the yearly censuses of the female black robin population between 1972 and 1998 ($N=26$).
The data are presented in Figure~\ref{fbr}, and the estimates are shown in Table~\ref{t_appBR}. Due to relatively large population sizes, obtaining the  MLE was numerically problematic.
The computation of the GW estimators is  305 times faster than for the SPMLEs, whose computation is iteself 167 times faster  on average than the MLEs. The GW estimates are very close to the MLEs, but the SPMLEs are a little different owing to small population sizes in the early years. The adjusted SPMLEs (correcting for the smoothing effect around small population sizes) show a substantial improvement.
 
We conclude that a one-year-old female bird has a life expectancy of $1/0.19=5.26$ years and gives birth on average to $0.318/0.19=1.67$ female offspring reaching an age of one year; these are reasonable estimates. Assuming constant birth and death rates, the estimated probability of extinction for the black robin is $0.19/0.318=0.6$ (SE $0.018$).


\begin{table}
\centering
\small{
\begin{tabular}{l|cccc}
&GW estimates & SPMLE& Adjusted SPMLE&  MLE \\\hline
$\hat{\lambda}$ &0.319 (0.071) & 0.301 (0.073) & 0.319 (0.072)   & 0.318 (0.072) \\
$\hat{\mu}$ &0.191 (0.071) &0.173 (0.072) & 0.191 (0.071)  &0.190 (0.071) \\
$\hat{\omega}$& 0.128 (0.028) & 0.128 (0.027)& 0.128 (0.028) & 0.128 (0.028) \\\hline CPU time (sec)  & $2.67\,10^{-4}$ & 0.081 &1.35 & 13.61
\end{tabular}}
\caption{\label{t_appBR}Estimates (and standard errors based on Theorem \ref{lambda&mu_gw_dist} for the GW estimates, and on Fisher Information for the SPMLE and the MLE) of the birth, death, and growth rates of the black robin population between 1972 and 1998 from the GW, SPMLE and MLE approaches, and average CPU times.}
\end{table}

\subsection{Whooping crane population}
The whooping crane is a rare migratory bird that breeds in northern Canada and winters in Texas. Stratton \cite{stratton} provides annual counts of whooping cranes arriving in Texas during autumn from 1938--2007 ($N=69$). LBDPs were used to model their population between 1938 and 1972 in Miller \textit{et al.} \cite{miller1974whooping} and in Guttorp \cite[p.47]{guttorp91}, but as the original dataset includes both males and females, dependences and mating make it unlikely that an LBDP fits properly. We therefore model the female population only, assuming a sex ratio of 1:1; see Figure~\ref{fbr}, and  Table~\ref{t_appWC2}. 
Again due to relatively large population sizes, the computation of the  MLEs caused numerical issues. The SPMLEs are closer to the  MLEs than to the GW estimates. The adjusted SPMLE gives no strong improvement because the population sizes are not close to zero. The estimated birth and death rates seem more realistic than those obtained in \cite{guttorp91} and give an  estimated extinction probability of $0.149/0.193=0.772$ (SE $0.002$).



\begin{table}
\centering
\begin{tabular}{l|cccc}
&GW estimates & SPMLE& Adjusted SPMLE&  MLE \\\hline
$\hat{\lambda}$ &0.186 (0.023) & 0.195 (0.031) & 0.195 (0.030)  & 0.193 (0.030) \\
$\hat{\mu}$ &0.142 (0.023) &0.152 (0.031)& 0.152 (0.030)  &0.149 (0.030) \\
$\hat{\omega}$& 0.044 (0.011) & 0.044 (0.011)& 0.044 (0.011) & 0.043 (0.011) 
\end{tabular}
\caption{\label{t_appWC2}Estimates (and standard errors computed as in Table \ref{t_appBR}) of the birth, death, and growth rates of the whooping crane female population.} 
\end{table}

\subsection{H1N1 influenza outbreak}

Finally, we apply our methods to the initial stage of the H1N1 pandemic in 2009 in Mexico,  using daily counts of new cases reported between the outbreak on 11~March and 24 April, when educational institutions in Mexico City were shut ($N=44$) \cite{fraser2009pandemic}. See Figure~\ref{fbr} and Table~\ref{t_appH1N1}. The SPMLEs are again closer to the MLEs than are the GW estimates, and adjusting for the smoothing effect again improves the result considerably. The estimated basic reproduction number is $\hat{R}_0=0.841/0.658\approx 1.28>1$, consistent with the Bayesian results of Fraser \textit{et al.} \cite{fraser2009pandemic}.  
The estimated Malthusian parameter, $0.044$, is lower than the estimate of $0.199$ obtained by Kraus and Panaretos \cite{andrea13} under the assumption of under-reporting. We estimate the expected length of the infectious period to be $1/0.15\approx 6.67$ days, which appears reasonable, since it is believed to be until five to seven days after the symptoms appear.


\begin{table}
\centering
\begin{tabular}{l|cccc}
&GW estimates & SPMLE&Adjusted SPMLE&   MLE \\\hline
$\hat{\lambda}$ &1.067 (0.208) & 0.952 (0.190) &0.883 (0.139)  & 0.841 (0.144) \\
$\hat{\mu}$ &0.884 (0.208) &0.770 (0.189) &0.701 (0.140) &0.658 (0.143) \\
$\hat{\omega}$& 0.182 (0.065) & 0.182 (0.061) & 0.182 (0.059) & 0.182 (0.057) 
\end{tabular}
\caption{\label{t_appH1N1}Estimates (and standard errors computed as in Table \ref{t_appBR}) of the birth (transmission), death (removal), and growth rates of the population infected by H1N1.} 
\end{table}

\bibliographystyle{agsm}


\section*{Acknowledgements}
 Sophie Hautphenne thanks the Australian Research Council for support through Discovery Early Career Researcher Award DE150101044. Andrea Kraus thanks the Czech Science Foundation for support through Grant GJ17-22950Y. The authors also thank Simon Tavar\'e and Phil Pollett for fruitful discussions.

\bigskip
\begin{center}
{\large\bf APPENDIX A: Proofs of theorems}
\end{center}

%
%
%
%


\noindent\textit{Proof of Theorem \ref{lambda&mu_gw_dist}.}
Consistency of $\hat\lambda_{Z_0, N}$ and $\hat\mu_{Z_0, N}$ follows readily from consistency of $\hat m_{Z_0, N}$ and $\widehat{\sigma^2}_{Z_0,N}$. 
To show asymptotic normality, we consider
\small{\begin{eqnarray}\nonumber
\hat\lambda_{Z_0, N}-\lambda
& = & 
\frac{\log (\hat m_{Z_0, N})}{2\Delta t} \left(\frac{\widehat{\sigma^2}_{Z_0, N}}{\hat m_{Z_0, N}(\hat m_{Z_0, N}-1)} +1 \right)
-
\frac{\log (m)}{2\Delta t} \left(\frac{\sigma^2}{m(m-1)} +1 \right)
\\
\nonumber & = & 
\frac{1}{2\Delta t}\times
\left\{
\log (\hat m_{Z_0, N}) 
-
\log (m)
\right\}
+
\frac{1}{2\Delta t}\times
\frac{\log (\hat m_{Z_0, N})}{\hat m_{Z_0, N}(\hat m_{Z_0, N}-1)}\times 
\left\{
\widehat{\sigma^2}_{Z_0, N} - \sigma^2
\right\}
\\\label{diff_lam}
& & \quad\quad + \quad
\frac{\sigma^2}{2\Delta t}\times 
\left\{
\frac{\log (\hat m_{Z_0, N})}{\hat m_{Z_0, N}(\hat m_{Z_0, N}-1)}
-
\frac{\log(m)}{m(m-1)}
\right\},
\end{eqnarray}}\normalsize
and similarly, we obtain 
\begin{eqnarray}\label{diff_mu}
\hat\mu_{Z_0, N} - \mu
& = & (\hat\lambda_{Z_0, N}-\lambda) - \frac{1}{\Delta t}\times
\left\{
\log (\hat m_{Z_0, N}) 
-
\log (m)
\right\}.
\end{eqnarray}
To obtain the result for $N\to\infty$, we recall~\eqref{m_gw_dist} and apply the delta method with $g_1(x)=\log(x)$ and $g_2(x) = \frac{\log(x)}{x(x-1)}$ to derive that
\begin{eqnarray*}
\sqrt{\sum_{n=1}^N Z_{n-1}}\, \left\{\log(\hat m_{Z_0, N}) - \log(m) \right\} 
& \xrightarrow[]{d} &
\mathcal{N}
\left(
0, \frac{\sigma^2}{m^2} 
\right)
,
\\
\sqrt{\sum_{n=1}^N Z_{n-1}}\, \left\{\frac{\log(\hat m_{Z_0, N})}{\hat m_{Z_0, N}(\hat m_{Z_0, N}-1)} - \frac{\log(m)}{m(m-1)} \right\} 
& \xrightarrow[]{d} &
\mathcal{N}
\left(
0, \sigma^2\, f(m)
\right)
\end{eqnarray*}
conditionally on $\{ Z_N\to\infty\}$ as $N\to\infty$,
where
$ f(m)=\{m-1-(2m-1)\log(m)\}^2/\{m^4(m-1)^4\}.$
Rewriting \eqref{diff_lam} as
\begin{small}
\begin{eqnarray}\nonumber\lefteqn{
2\Delta t\,\sqrt{N}\, \left( \hat\lambda_{Z_0, N} - \lambda \right)}\\
& = & \nonumber
\sqrt{\frac{N}{\sum_{n=1}^N Z_{n-1}}}\times\sqrt{\sum_{n=1}^N Z_{n-1}}
\left\{
\log (\hat m_{Z_0, N}) 
-
\log (m)
\right\}
\\\nonumber
& & + \quad
\sqrt{N}\,
\frac{\log (\hat m_{Z_0, N})}{\hat m_{Z_0, N}(\hat m_{Z_0, N}-1)}\times 
\left\{
\widehat{\sigma^2}_{Z_0, N} - \sigma^2
\right\}
\\
& &\label{diff_lam2}  + \quad
\sqrt{\frac{N}{\sum_{n=1}^N Z_{n-1}}}\,\sqrt{\sum_{n=1}^N Z_{n-1}}
\left\{
\frac{\log (\hat m_{Z_0, N})}{\hat m_{Z_0, N}(\hat m_{Z_0, N}-1)}
-
\frac{\log(m)}{m(m-1)}
\right\}
,\qquad
\end{eqnarray}
\end{small}and observing that
$N/(\sum_{n=1}^N Z_{n-1})
\to 0$ a.s. as $N\to\infty$ on  $Z_N\to\infty$ 
\cite{guttorp91}, we obtain that
the first and the third summand on the right-hand side of \eqref{diff_lam2} tend to zero in probability conditionally on $\{Z_N\to\infty\}$ as $N\to\infty$.
For the middle term, we use consistency of $\hat m_{Z_0, N}$ and~\eqref{sig_gw_dist} to derive that
\[
\sqrt{N}\,
\frac{\log (\hat m_{Z_0, N})}{\hat m_{Z_0, N}(\hat m_{Z_0, N}-1)}\times 
\left\{
\widehat{\sigma^2}_{Z_0, N} - \sigma^2
\right\}
\xrightarrow[]{d}
\mathcal{N}
\left(
0,
\frac{\{\log (m)\}^2}{m^2(m-1)^2}\, 2\sigma^4
\right)
\]
conditionally on $\{Z_N\to\infty\}$ as $N\to\infty$.
It follows that
\[
\sqrt{N}\, \left( \hat\lambda_{Z_0, N} - \lambda \right)
\xrightarrow[]{d}
\mathcal{N}
\left(
0,
\frac{\{\log (m)\}^2 \sigma^4}{2(\Delta t)^2 m^2(m-1)^2}
\right)
\]
conditionally on $\{Z_N\to\infty\}$ as $N\to\infty$.
Moreover, by \eqref{diff_mu},
\[
\sqrt{N}\, \left( 
\begin{matrix}
\hat\lambda_{Z_0, N} - \lambda  \\
\hat\mu_{Z_0, N} - \mu 
\end{matrix}
\right)
\xrightarrow[]{d}
\mathcal{N}
\left\{
\left(
\begin{matrix}
0 \\
0 \\
\end{matrix}
\right)
,
\frac{\{\log (m)\}^2 \sigma^4}{2(\Delta t)^2 m^2(m-1)^2}
\times
\left(
\begin{matrix}
1 & 1 \\
1 & 1 \\
\end{matrix}
\right)
\right\}
\]
conditionally on $\{Z_N\to\infty\}$ as $N\to\infty$.

To obtain the asymptotic result for $N\to\infty$ and $Z_0\to\infty$, we recall~\eqref{m&sig_gw_dist} and again apply the delta method with $g_1(x)$ and $g_2(x) $, leading to
\begin{eqnarray*}
\sqrt{Z_0\,(m^N-1)}\, \left\{\log(\hat m_{Z_0, N}) - \log(m) \right\} 
& \xrightarrow[]{d} &
\mathcal{N}
\left(
0, \sigma^2\, \frac{(m-1)}{m^2} 
\right),
\\
\sqrt{Z_0\,(m^N-1)}\, \left\{\frac{\log(\hat m_{Z_0, N})}{\hat m_{Z_0, N}(\hat m_{Z_0, N}-1)} - \frac{\log(m)}{m(m-1)} \right\} 
& \xrightarrow[]{d} &
\mathcal{N}
\left(
0, \sigma^2\, (m-1) f(m)
\right), 
\end{eqnarray*}
conditionally on $\{ Z_N\to\infty\}$ as $Z_0\to\infty$ and $N\to\infty$.
Rewriting \eqref{diff_lam} as
\begin{small}
\begin{eqnarray}\nonumber\lefteqn{
2\Delta t\,\sqrt{N}\, \left( \hat\lambda_{Z_0, N} - \lambda \right)
}\\& = & \nonumber
\sqrt{\frac{N}{Z_0\,(m^N-1)}}\,\sqrt{Z_0\,(m^N-1)}
\left\{
\log (\hat m_{Z_0, N}) 
-
\log (m)
\right\}
\\\nonumber
& &  + \quad
\sqrt{N}\,
\frac{\log (\hat m_{Z_0, N})}{\hat m_{Z_0, N}(\hat m_{Z_0, N}-1)}\times 
\left\{
\widehat{\sigma^2}_{Z_0, N} - \sigma^2
\right\}
\\
& & \label{diff_lam3} + \quad
\sqrt{\frac{N}{Z_0\,(m^N-1)}}\,\sqrt{Z_0\,(m^N-1)}
\left\{
\frac{\log (\hat m_{Z_0, N})}{\hat m_{Z_0, N}(\hat m_{Z_0, N}-1)}
-
\frac{\log(m)}{m(m-1)}
\right\},\qquad
\end{eqnarray}
\end{small}and observing that 
$N/Z_0 (m^N-1)\to 0$ 
as $N\to\infty$, and also as $N\to\infty$ and $Z_0\to\infty$, we obtain that
the first and the third summands on the right-hand side of \eqref{diff_lam3} tend to zero in probability conditional on $\{Z_N\to\infty\}$ as $N\to\infty$ and $Z_0\to\infty$.
For the middle term, using consistency of $\hat m_{Z_0, N}$, \eqref{m&sig_gw_dist}, and the same argument as earlier when $N\to\infty$, we arrive at the same result as $N\to\infty$ and $Z_0\to\infty$, conditionally on $\{Z_N\to\infty\}$.
\hfill $\square$

\bigskip

\noindent \textit{Proof of Lemma \ref{explicit}.}
We show the result for the case $\lambda\neq \mu$, the case $\lambda= \mu$ being similar. 
Setting $K(x,t):=K(x,t;1)$, the explicit expression for $f(s,t)$ leads to
\begin{eqnarray*}\frac{\partial}{\partial x}K(x,t)&=&\frac{\partial}{\partial x} \log f(\exp(\omega x),t)\;=\;f(\exp(\omega x),t)^{-1}\, \frac{\partial}{\partial s}f(s,t)\Big|_{s=\exp(\omega x)} \,\exp(\omega x) \qquad\\[1em]&=& \dfrac{m(t) (\lambda - \mu)^2 s}{\{\mu-\lambda s +\lambda(s-1)\, m(t)\}\{\mu-\lambda s +\mu(s-1)\, m(t)\} }\Big|_{s=\exp(\omega x)}. \end{eqnarray*} Solving for the saddle point $\tilde{x}=\tilde{x}(k,t;a)$ satisfying \eqref{saddle_equ1} then reduces to solving for the unique solution $\tilde{s}=\tilde{s}(k,t;a)$ lying in $[0,R(t))$ to the second degree equation
\begin{equation}\label{spe_bd}\dfrac{m(t) (\lambda - \mu)^2 s}{\{\mu-\lambda s +\lambda(s-1)\, m(t)\}\{\mu-\lambda s +\mu(s-1)\, m(t)\} }=\dfrac{k}{a},\end{equation}
 and setting $\tilde{x}=\log \tilde{s}$. 
 Equation \eqref{spe_bd} can be rewritten as $Q(s):=A s^2+B s +C=0$ with solutions $s_{\pm}=(-B\pm \sqrt{B^2-4 AC})/(2A)$, where 
 \begin{eqnarray*}A:=A(t)&=&\lambda\{m(t)-1\}\{\lambda-\mu m(t)\},\\
B:=B(k,t;a)&=&2\lambda\mu\{1+m(t)^2-m(t)-(a/k)m(t)\}+m(t)(\lambda^2+\mu^2)\{(a/k)-1\},\\
 C:=C(t)&=&\mu\{m(t)-1\}\{\mu-\lambda m(t)\}.\end{eqnarray*}
 If $A>0$, then $s_{-}\leq s_{+}$; in addition, we have $Q(0)=C<0$ (to see this, note that $R(t)=\{\lambda m(t)-\mu\}/\{\lambda(m(t)-1)\}>0$), which implies that $s_{-}<0$ and therefore $s_{+}$ is the saddlepoint. If $A>0$, then $s_{+}\leq s_{-}$; in addition, we have $Q(R(t))=(a/k) m(t) (\lambda-\mu)^2 \,R(t)>0$, which implies that $s_{-}>R(t)$ and therefore $s_{+}$ is the saddlepoint. We conclude that $\tilde{s}=s_{+}$.
 
 Finally, setting $K''(x,t):=\partial^2 K(x,t)/\partial x^2$, we have
 \begin{eqnarray*}K''(\tilde{x},t)=&=&-f(\tilde{s},t)^{-2}\, \left\{\frac{\partial}{\partial s}f(s,t)\Big|_{s=\tilde{s}} \,\tilde{s}\right\}^2+f(\tilde{s},t)^{-1}\, \frac{\partial^2}{\partial s^2}f(s,t)\Big|_{s=\tilde{s}} \,\tilde{s}^2\\&&+f(\tilde{s},t)^{-1}\, \frac{\partial}{\partial s}f(s,t)\Big|_{s=\tilde{s}} \,\tilde{s}\\[1em]&=&-\frac{\{m(t)-1\} m(t) \tilde{s} (\lambda-\mu)^2 \left\{-\lambda^2 \tilde{s}^2+\lambda m(t) \mu \left(\tilde{s}^2-1\right)+\mu^2\right\}}{\{\lambda [m(t) (\tilde{s}-1)-\tilde{s}]+\mu\}^2 \{\lambda \tilde{s}+\mu [-m(t) \tilde{s}+m(t)-1]\}^2} .\end{eqnarray*}
 Using \eqref{spa_bd_a} and the above expressions, we then obtain \eqref{spa_bd3}.
 \hfill $\square$

\bigskip

\noindent\emph{Proof of Lemma \ref{SPA_GW}}. Let $\Delta t$ be the inter-observation time.  Using \eqref{spaGauss} and the fact that $m(\Delta t;a)=a m(\Delta t;1):=a m$ and $\sigma^2(\Delta t;a)=a \sigma^2(\Delta t;1):=a \sigma^2$, we obtain an approximation for saddlepoint log-likelihood function \eqref{spall}:
\begin{small}
\[
\tilde{\tilde{\ell}}(\lambda,\mu; \vc t,\vc k)= -\frac{1}{2}\sum_{i=1}^M\sum_{j=1}^N  \log(2\pi \sigma^2 k_{i, j-1} ) -\frac{1}{2\sigma^2}\sum_{i=1}^M \sum_{j=1}^N \dfrac{(k_{i, j}-k_{i, j-1}\, m)^2}{k_{i, j-1}},
\] 
\end{small}where $m$ and $\sigma^2$ are functions of $\lambda$ and $\mu$.
We then have
\begin{eqnarray*}
\frac{\partial}{\partial m}\tilde{\ell}(\lambda,\mu; \vc t,\vc k)
& = & 
 \frac{1}{\sigma^2}\sum_{i=1}^M\sum_{j=1}^N (k_{i, j}-k_{i, j-1} m),
\\
\frac{\partial}{\partial \sigma^2}\tilde{\ell}(\lambda,\mu; \vc t,\vc k)
& = & 
-\frac{N M}{2\sigma^2} + \frac{1}{2\sigma^4}\sum_{i=1}^M\sum_{j=1}^N \dfrac{(k_{i, j}-k_{i, j-1}\, m)^2}{k_{i, j-1}},
\end{eqnarray*}
so the MLEs of $m$ and $\sigma^2$ based on \eqref{spaGauss} are
\begin{eqnarray*}
\hat{\tilde{m}}(\Delta t;1) & = & \frac{\sum_{i=1}^M\sum_{j=1}^N k_{i, j}}{\sum_{i=1}^M\sum_{j=1}^N k_{i, j-1}};
\\
\widehat{\widetilde{\sigma^2}}(\Delta t;1)
& = & 
\frac{1}{N M}\sum_{i=1}^M \sum_{j=1}^N \dfrac{(k_{i, j}-k_{i, j-1} \hat{\tilde{m}})^2}{k_{i, j-1}} = 
\frac{1}{N M}\sum_{j=1}^N k_{i, j-1} \left(\dfrac{k_{i, j}}{k_{i, j-1}} - \hat{\tilde{m}}\right)^2,
\end{eqnarray*}
which coincide with the estimators~\eqref{m_GW_M} of the offspring mean and variance of the embedded GW process. \hfill $\square$

\bigskip

\noindent \emph{Proof of Theorem~\ref{GASPAthm}}.
We use Theorems~5.41 and~5.42 of van der Vaart \cite{vanderVaart:1998}, in which $\vc\theta$ is the vector of model parameters, $\psi_\theta(\cdot)$ is the theoretical score vector based on the log-likelihood $\ell(\vc\theta; \cdot)$, and $\Psi(\vc\theta)$ is the empirical score vector.  The proof has two main parts.   We first verify under what circumstances the conditions of Theorem~5.41 hold, and then we show that the profile log-likelihood for $\omega$ has a unique maximum.  This implies that if the score equations have a solution, it must be unique and therefore gives the consistent estimator, by Theorem~5.42. 

\paragraph{Part 1}

Consider a single replicate of the LBDP and drop the subscript $i$. 

Let $\{Z(t): t\geq 0\}$ be a stochastic process with $Z(0)=a>0$ known, and suppose that $Z(t)$ is observed subsequently at times $t_1=\tau_1$, $t_2=\tau_1+\tau_2$, etc., where all the $\tau_i>\v$; write $t_0=0$.  We abuse  notation and write $Z_j=Z(t_j)$.  Let $\F=\{\F_j, j\in\Natural\}$ denote the filtration generated by $\{Z_j, j\in\Natural\}$, conditional on $Z(0)=a$.  Conditional on $\F_{j-1}$, suppose that $Z_j$ has mean and variance  
\begin{equation}
\label{ar1.eq}
k_{j-1} \zeta_j(\omega):=k_{j-1}\exp(\omega \tau_j),\quad k_{j-1}\xi\nu_j(\omega) := k_{j-1}\xi\tau_j\zeta_j(\omega)/c(\omega\tau_j),
\end{equation}
with $c(u) = u/\{\exp(\omega u)-1\}$ strictly monotone decreasing for $u\in\Real$.  Calculations in \cite[Example~5.4]{davison2003statistical} imply that $\kappa(u) = -\log c(u)$ is strictly convex, $\kappa'(u)$ is monotone increasing with limits 0 and 1, and $\kappa''(u)>0$.

Write $\vc\theta=(\xi,\omega)$ and let $\vc\theta_0=(\xi_0,\omega_0)$ denote the value of $\vc\theta$ that generated the data; note that $\vc\theta\in\Theta=\{(\xi,\omega): \xi+\omega>0, \xi-\omega>0\}$ and that $\Theta$ is an open subset of $\Real^2$.   Below we shall write $\dot\zeta_j=\D{\zeta_j}(\omega)/\D{\omega}$ and so forth, use $\zeta_j^0$, $\dot\zeta_j^0$ and so forth to indicate quantities evaluated at $\vc\theta_0$, and use $\E_0$ to indicate expectation at $\vc\theta=\vc\theta_0$; $\E_0(Z_j)= a\exp(\omega_0 t_j)$.

Suppose that observation of $\{Z_j\}$  stops either after a finite number of generations, $N\geq 1$, or when the population becomes extinct, i.e., at time $T=\min(N, \min\{ t_j:Z_j=0\})$.  If $T$ is finite, this implies that even if $Z_T=0$, all previous values of $Z_j$ are positive.

Estimation is based on supposing that, 
 conditionally on $Z_{j-1}=k_{j-1}$ (with $k_0:=a$), $Z_j$ is normally distributed with mean and variance given by~\eqref{ar1.eq}, for $1\leq j\leq T$.
Apart from an additive constant, this gives the log-likelihood $\logL(\vc\theta)=\sum_{j=1}^T \logL_j(\vc\theta)$, with
$$
\logL_j(\vc\theta) \equiv \logL_j(\vc\theta; k_j, k_{j-1}) =  -{1\over 2}\left\{\log(k_{j-1}\xi \nu_j) + {(k_j-k_{j-1}\zeta_j)^2\over k_{j-1}\xi\nu_j}\right\}, \quad \vc\theta\in\Theta;
$$
here and below we suppress the dependence of $\zeta_j$ and $\nu_j$ on $\omega$. 
This is the true log-likelihood if the assumption of conditional normality is correct, and a quasi-likelihood for estimation of $\vc\theta$ when the mean and variance functions in~\eqref{ar1.eq} are correctly specified; the third and fourth cumulants of $(Z_j-\zeta_j^0Z_{j-1})/(\xi_0\nu_j^0Z_{j-1})^{1/2}$ conditional on $\F_{j-1}$ appear when the increments are non-normal, but these cumulants are finite by hypothesis.  

We now verify that the conditions of Theorem~5.41. 
The $2\times1$  score vector $\psi_\theta(x)$ based on $\logL(\vc\theta)$ has elements 
\begin{eqnarray}
\label{score-xi.eqn}
{\partial\logL(\vc\theta)\over\partial \xi} &=&  -{1\over 2}\sum_{j=1}^T \left\{{1\over \xi} - {(k_j-k_{j-1}\zeta_j)^2\over \xi^2 k_{j-1}\nu_j}\right\},\\
\label{score-omega.eqn}
{\partial\logL(\vc\theta)\over\partial \omega} &=&  -{1\over 2}\sum_{j=1}^T \left\{{\dot\nu_j\over\nu_j} -  {\dot\nu_j\over\nu_j}{(k_j-k_{j-1}\zeta_j)^2\over \xi k_{j-1}\nu_j} - 2{\dot\zeta_j\over \xi\nu_j} (k_j-k_{j-1}\zeta_j)\right\};
\end{eqnarray}
clearly these are twice continuously differentiable for $\vc\theta\in\Theta$ for every $z=(k_0,\ldots, k_T)$. 

To prove that the components of the score vector have zero expectations, note that 
$$
A_J := \sum_{j=1}^J \left\{{1\over \xi_0} - {(Z_j-Z_{j-1}\zeta^0_j)^2\over \xi_0^2 Z_{j-1}\nu^0_j}\right\},\qquad 1\leq J\leq N,
$$
satisfies $\E_0(A_J\mid \F_{J-1})=A_{J-1}$.  Now    
\begin{eqnarray*}
\E_0(|A_J|)&\leq &\sum_{j=1}^J \E_0\left\{\left|{1\over \xi_0} - {(Z_j-Z_{j-1}\zeta^0_j)^2\over Z_{j-1}\xi_0^2\nu^0_j}\right|\right\}\\
&\leq & {J\over \xi_0} + \sum_{j=1}^J \E_0\left\{ {(Z_j-Z_{j-1}\zeta^0_j)^2\over Z_{j-1}\xi_0^2\nu^0_j}\right\}  = 2J/\xi_0,
\end{eqnarray*}
and as $\E_0(|A_J|)<\infty$, we see that $\{A_j\}$ is a martingale with respect to $\F$.  Moreover,  $T$ is a stopping time, and as $T\leq N$, we have that $\pr_0(T<\infty)=1$, $\E(|A_T|)<\infty$ and $\E\{A_j \mathds{1}_{\{T>t\}}\}\to0$ as $t\to\infty$. Hence the optional stopping theorem \cite[Theorem 12.5.1]{Grimmett.Stirzaker:2001} applies, and thus $\E(A_T)=\E(A_1)=0$.  
This argument also applies to $\partial\logL(\vc\theta)/\partial\omega$, so $\E_0\{\partial\logL(\vc\theta^0)/\partial\xi\}=\E_0\{\partial\logL(\vc\theta^0)/\partial\omega\}=0$, as required. 

We now need to check that the covariance matrix $\vc{\mathrm{C}}(\vc\theta)$ of the score vector is finite. Conditional independence of the summands of the score yields that at $\vc\theta=\vc\theta_0$, 
\small{\begin{eqnarray*}
\var_0\left\{{\partial\logL(\vc\theta)\over\partial \xi} \right\}  &= &{1\over 4} \E_0\left[\sum_{j=1}^T \left\{{1\over \xi_0} - {(Z_j-Z_{j-1}\zeta^0_j)^2\over \xi_0^2 Z_{j-1}\nu^0_j}\right\}^2\right] ,\\
\var_0\left\{{\partial\logL(\vc\theta)\over\partial \omega} \right\}  &= &{1\over 4} \E_0\left[\sum_{j=1}^T \left\{{\dot\nu^0_j\over\nu^0_j} -  {\dot\nu^0_j\over\nu^0_j}{(Z_j-Z_{j-1}\zeta^0_j)^2\over \xi_0 Z_{j-1}\nu^0_j} - 2{\dot\zeta^0_j\over \xi_0\nu^0_j} (Z_j-Z_{j-1}\zeta^0_j)\right\}^2\right] ,\\
\cov_0\left\{{\partial\logL(\vc\theta)\over\partial \xi}, {\partial\logL(\vc\theta)\over\partial \omega}\right\}  &= &-{1\over 4\xi_0}  \E_0\left[\sum_{j=1}^T\left\{{\dot\nu^0_j\over\nu^0_j}- {\dot\nu^0_j\over\nu^0_j}{(Z_j-Z_{j-1}\zeta^0_j)^2\over \xi_0 Z_{j-1}\nu^0_j} - 2{\dot\zeta^0_j\over \xi_0\nu^0_j} (Z_j-Z_{j-1}\zeta^0_j)\right\} \right.\\
&&\qquad\qquad\qquad\quad \left. \times{(Z_j-Z_{j-1}\zeta^0_j)^2\over \xi_0 Z_{j-1}\nu^0_j}\right] ,
\end{eqnarray*}}\normalsize
and some calculation shows that $\vc{\mathrm{C}}(\vc\theta_0)$ has components 
\small{\begin{eqnarray*}
\label{cov.eq}
\var_0\left\{{\partial\logL(\vc\theta)\over\partial \xi} \right\}  &= &{1\over 2\xi^2_0}\E_0(T) + {1\over 4\xi_0^2} \E_0\left(\sum_{j=1}^T \kappa^0_{4,t}\right),\\
\var_0\left\{{\partial\logL(\vc\theta)\over\partial \omega} \right\} &=&   \E_0\left\{{1\over 4}\sum_{j=1}^T \left({\dot\nu_j^0\over \nu_j^0}\right)^2(2+\kappa^0_{4,t}) + 
{a\over \xi_0} \sum_{j=1}^T {(\dot\zeta_j^0)^2\over \nu_j^0}\exp(\omega_0 t_{j-1})\right.\\&&\qquad\qquad\qquad\quad \left. + \sum_{j=1}^T {\dot\nu^0_j\dot\zeta_j^0Z_{j-1}^{1/2}\over \{\xi_0(\nu_j^0)^3\}^{1/2}}\kappa^0_{3,t}\right\},\\
\cov_0\left\{{\partial\logL(\vc\theta)\over\partial \xi}, {\partial\logL(\vc\theta)\over\partial \omega}\right\}  &=&{1\over 4\xi_0}\E_0\left[ \sum_{j=1}^T \left\{{\dot\nu^0_j\over\nu^0_j}(2+\kappa^0_{4,t})+ 2{\dot\zeta^0_jZ_{j-1}^{1/2}\over (\xi_0\nu^0_j)^{1/2}}\kappa^0_{3,t}\right\}\right], 
\end{eqnarray*}}\normalsize
where $\kappa^0_{3,t}$ and $\kappa^0_{4,t}$ are the third and fourth cumulants of $(Z_j-\zeta_j^0Z_{j-1})/(\xi_0\nu_j^0Z_{j-1})^{1/2}$ conditional on $\F_{j-1}$, when $\vc\theta=\vc\theta_0$; these are finite by hypothesis, and since $T\leq N$, the components of $\vc{\mathrm{C}}(\vc\theta_0)$ are finite. The expressions simplify when $\kappa^0_{3,t}=\kappa^0_{4,t}\equiv0$.

To show that the information matrix $\vc{\mathrm{I}}(\vc\theta)$ exists and is non-singular, note that
\begin{eqnarray*}
{\partial^2\logL(\vc\theta)\over\partial \xi^2} &=&  -{1\over 2}\sum_{j=1}^T \left\{2  {(k_j-k_{j-1}\zeta_j)^2\over \xi^3 k_{j-1}\nu_j}- {1\over \xi^2} \right\},\\
{\partial^2\logL(\vc\theta)\over\partial\xi\partial \omega} &=&  -{1\over 2\xi^2}\sum_{j=1}^T \left\{   {\dot\nu_j\over\nu^2_j}{(k_j-k_{j-1}\zeta_j)^2\over k_{j-1}\nu_j}  + 2{\dot\zeta_j\over\nu_j} (k_j-k_{j-1}\zeta_j)\right\},\\
{\partial^2\logL(\vc\theta)\over\partial \omega^2} &=&  -{1\over 2}\sum_{j=1}^T \left\{   {\ddot\nu_j\over\nu_j} - {\dot\nu^2_j\over\nu^2_j} - {\ddot\nu_j\over\nu^2_j}{(k_j-k_{j-1}\zeta_j)^2\over \xi k_{j-1}\nu_j}  + 2{\dot\nu^2_j\over\nu^3_j} {(k_j-k_{j-1}\zeta_j)^2\over \xi k_{j-1}}\right.\\
&&\qquad\qquad\quad  + \,\left. 4{\dot\nu_j\dot\zeta_j\over \xi \nu_j^2}(k_j-k_{j-1}\zeta_j) - 2{\ddot\zeta_j\over \xi\nu_j} (k_j-k_{j-1}\zeta_j) + 2k_{j-1}{\dot\zeta_j^2\over \xi\nu_j}\right\}.
\end{eqnarray*}
On replacing $k_j$ and $k_{j-1}$ in these expressions with the corresponding random variables $Z_j$ and $Z_{j-1}$, it is straightforward to check that when $\vc\theta=\vc\theta_0$, 
\begin{eqnarray*}
\E_0\left\{-{\partial^2\logL(\vc\theta)\over\partial \xi^2} \right\} &=& {\E_0(T)\over 2\xi_0^2}, \quad \E_0\left\{-{\partial^2\logL(\vc\theta)\over\partial\xi\partial \omega}\right\}= {1\over 2\xi_0} \E_0\left(\sum_{j=1}^T {\dot\nu_j^0\over \nu_j^0}\right), \\
\E_0\left\{-{\partial^2\logL(\vc\theta)\over\partial \omega^2}\right\}&=& {1\over 2} \E_0\left\{\sum_{j=1}^T \left({\dot\nu_j^0\over \nu_j^0}\right)^2 + {2a\over \xi_0} \sum_{j=1}^T {(\dot\zeta_j^0)^2\over \nu_j^0}\exp(\omega_0t_{j-1})\right\}.
\end{eqnarray*}
If $\kappa^0_{3,t}=\kappa^0_{4,t}\equiv 0$, then $\vc{\mathrm{C}}(\vc\theta_0)=\vc{\mathrm{I}}(\vc\theta_0)$.   The two components of the score vector are linearly independent, so $\vc{\mathrm{I}}(\vc\theta_0)$ is  positive definite.  As $\vc\theta_0$ is necessarily interior to $\Theta$, there clearly exists a neighbourhood of $\vc\theta_0$ in which the third derivatives of $\logL(\vc\theta)$ are continuous in $\vc\theta$ and can be bounded by integrable functions of the data. 

We have verified the conditions for Theorem~5.41 of \cite{vanderVaart:1998}, and it follows that if we take $M$ independent copies of the above process and let $M\to\infty$, then the corresponding solutions to the score equations, if consistent, will have a limiting normal distribution with mean $\vc\theta_0$ and covariance matrix $\vc{\mathrm{I}}(\vc\theta_0)^{-1}\vc{\mathrm{C}}(\vc\theta_0)\vc{\mathrm{I}}(\vc\theta_0)^{-1}$, provided that the sequences of observation times $t_{i,1},\ldots, t_{i,T_i}$ for an infinite number of these copies are non-trivial, as is the case if all the $\tau_{i,j}>\v$ for some positive $\v$. 

\paragraph{Part 2}

We now show that the solution to the score equation $\pr_n\psi_\theta=0$ is consistent. To lighten notation we first consider a single set of data, $k_0=a,\ldots, k_T$.  Note from~\eqref{score-xi.eqn}  that setting $\partial\logL(\vc\theta)/\partial\xi=0$ implies that the unique maximum of $\logL$ with respect to $\xi$ for each $\omega$ is at
$\hat\xi_\omega= T^{-1}\sum_{j=1}^T (k_j-k_{j-1}\zeta_j)^2/(k_{j-1}\nu_j)$. Hence the overall maximum of $\logL$  in terms of $\omega$ is obtained by maximising the profile log-likelihood 
$$
\logLp(\omega) = \logL(\hat\xi_\omega,\omega) \equiv -{T\over 2}\log\hat\xi_\omega  - {1\over 2}\sum_{j=1}^T \log\nu_j \equiv - {T\over 2} \log \left\{ \sum_{j=1}^T {(k_j-k_{j-1}\zeta_j)^2\over k_{j-1}\nu_j} \right\}- {1\over 2}\sum_{j=1}^T \log\nu_j,
$$
where we have dropped irrelevant constants. If we assume that a unique value $\hat\omega$ of $\omega$ maximises $\logLp(\omega)$, then it satisfies the equation
$$
0 = {\D{} \logLp(\omega)\over\D{\omega}} =
{\partial \hat\xi_\omega\over \partial\omega} \left.{\partial \logL(\xi,\omega)\over\partial\xi}\right|_{\xi=\hat\xi_\omega} + \left.{\partial \logL(\xi,\omega)\over\partial\omega}\right|_{\xi=\hat\xi_\omega} 
={\partial \logL(\hat\xi_\omega,\omega)\over\partial\omega}.
$$
Hence solving the score equation $\pr_n\psi_\theta=0$ and maximising $\logLp(\omega)$ with respect to $\omega$ give the same solutions $\hat{\vc\theta} = (\hat\xi_{\hat\omega},\hat\omega)$.

To prove the uniqueness of $\hat\omega$, note that the second term of $-2\logLp(\omega)$  has derivative 
\begin{equation}
\label{first-deriv.eq}
{\D{}\over \D{\omega}} \sum_{j=1}^T \log\nu_j={\D{}\over \D{\omega}} \sum_{j=1}^T \left\{\log\tau_j + \omega\tau_j +\kappa(\omega\tau_j)\right\}  = \sum_{j=1}^T \left\{\tau_j + \tau_j\kappa'(\omega\tau_j)\right\},
\end{equation}
which is strictly monotone increasing but bounded, because $\kappa'(u)\in (0,1)$ for all real $u$.

If $k_j>0$, one can check that ${(k_j-k_{j-1}\zeta_j)^2/( k_{j-1}\nu_j) }= {4k_j\tau^{-1}_j}g(\omega\tau_j,b_j)$,
where $b_j=\log(k_j/k_{j-1})$ and $g(u,b) = c(u) \sinh^2\{(u-b)/2\}$, for some real $b$.  
Now $c(u)$ is positive and strictly monotone decreasing, and $\sinh^2\{(u-b)/2\}$ has a unique minimum at $u=b$, so $g(u,b)$ has a global minimum $g(b,b)=0$ at $u=b$.  Moreover 
\begin{eqnarray*}
g'(u,b) &=& a'(u)\sinh^2\{(u-b)/2\} + c(u)\sinh\{(u-b)/2\}\cosh\{(u-b)/2\}\\
& =& g(u,b) \left[c'(u)/c(u) +\cosh\{(u-b)/2\}/\sinh\{(u-b)/2\} \right]\\
& =& g(u,b) \left[\cosh\{(u-b)/2\}/\sinh\{(u-b)/2\} -\kappa'(u)\right],
\end{eqnarray*}
so $g'(u,b)=0$ when $u=b$, and since $0<\kappa'(u)<1$, we see that $g'(u,b)<0$ for all $u<b$.  For $u>b$, $\cosh\{(u-b)/2\}/\sinh\{(u-b)/2>1$, so there are no roots of $g'(u,b)$ when $u>b$.  Hence $g(u,b)$ has a unique minimum at $u=b$, and $g'(u,b)$ is strictly monotone increasing.  

If $k_j=0$, then $j=T$ and 
$$
{(k_j-k_{j-1}\zeta_j)^2\over k_{j-1}\nu_j }= {k_{j-1}\over\tau_T}\exp(\omega\tau_T)c(\omega\tau_T) = {k_{j-1}\over\tau_T}c(-\omega\tau_T),
$$
which is positive and strictly monotone increasing in $\omega$; its derivative increases from 0 to $\infty$ as $\omega$ traverses the real line. Together with the results for $t=1,\ldots, T-1$, we find that 
$T\hat\xi_\omega = \sum_{j=1}^T {(k_j-k_{j-1}\zeta_j)^2/(k_{j-1}\nu_j)}$ 
has a unique, global, minimum as a function of $\omega$. Clearly the same is true of $\log \hat\xi_\omega$, whose derivative with respect to $\omega$ is therefore monotone increasing.  As this is also true of~\eqref{first-deriv.eq}, any solution $\hat\omega$ to $\D\logLp(\omega)/\D{\omega}=0$ is unique.  

If there are $M$ independent replicates with components $k_{i,0},\ldots, k_{i,T_i}$, where $T_i$ is the stopping time for the $i$th replicate and $i=1,\ldots, M$, then the log-likelihood may be written as $\logL(\vc\theta)=\sum_{i=1}^M \sum_{j=1}^{T_i} \logL_j(\vc\theta; k_{i,j}, k_{i,j-1})$. Apart from additive constants, 
$-2\logLp(\omega)=\sum_{i=1}^M T_i \log \hat \xi_\omega  + \sum_{i=1}^M \sum_{j=1}^{T_i} \log\nu_{i,j}$, where 
$$
\hat \xi_\omega = {1\over \sum_{i=1}^M T_i} \sum_{i=1}^M \sum_{j=1}^{T_i} {(k_{i,j}-k_{i,j-1}\zeta_{i,j})^2\over k_{i,j-1}\nu_{i,j}},\quad \zeta_{i,j} = \exp(\tau_{i,j}\omega), \quad \nu_{i,j} = {\tau_{i,j}\zeta_{i,j}\over c(\omega\tau_{i,j})},
$$
and the above argument for uniqueness of $\hat\omega$ still applies. 
Under the conditions of Theorem 5.41, $\hat\omega$ is therefore a consistent estimator of $\omega$ as $M\to\infty$.  Since $\hat\xi_{\hat \omega}$ is the unique corresponding estimator of $\xi$, it too is consistent.

The continuous mapping theorem and delta method now yield the consistency and joint asymptotic normality of $\lambda=(\xi+\omega)/2$ and $\mu=(\xi-\omega)/2$, with asymptotic covariance matrix given by
\begin{equation}\label{D.eq}
\var\begin{pmatrix} \hat\lambda\\ \hat\mu\\ \end{pmatrix}=\vc{\mathrm{D}}\vc{\mathrm{I}}(\vc\theta_0)^{-1}\vc{\mathrm{C}}(\vc\theta_0)\vc{\mathrm{I}}(\vc\theta_0)^{-1}\vc{\mathrm{D}}, \quad \vc{\mathrm{D}} = \dfrac{1}{2}\begin{pmatrix} 1& 1\\ 1& -1\\\end{pmatrix}.
\end{equation}\hfill $\square$

\bigskip
\begin{center}
{\large\bf APPENDIX B: {Multivariate saddlepoint}}
\end{center}

Instead of using the decomposition of the log-likelihood as a sum of univariate probabilities in \eqref{eq2}, we can alternatively work with the multivariate form \eqref{eq1} of the log-likelihood.
For any $N$ and observation times $t_0\leq t_1\leq t_2\leq  \cdots\leq t_N$, the conditional joint distribution of the population sizes $(Z(t_1),\ldots,Z(t_N))$, given $Z(t_0)= a$, has the multivariate PGF
$$g(\vc s,\vc t;a):=
f \{s_1\,f\{s_2\, \cdots f\{s_{N-1}\,f(s_N,\tau_N),\tau_{N-1}\}\cdots,\tau_2\},\tau_1\}^{ a},$$where $\vc s=(s_1,\ldots,s_N)$ and $\tau_{j}:=t_{j}-t_{j-1}$.
Explicit expressions for the first and second partial derivatives of $g(\vc s,\vc t;a)$ with respect to the entries of $\vc s$ can be obtained but are more cumbersome than in the univariate case.

%

Similar to the univariate case, we define the multivariate probability mass function
$${p}_{\vc k}(\vc t; a):=\mathsf{P}\{(Z(t_1),\ldots,Z(t_N))=\vc k\,|\,Z(0)=a\},\quad \vc k=(k_1,\ldots,k_N)\geq \vc 0.$$ We let
$K(\vc x,t;  a)=\log g(e^{\vc x},t;a)=a K(\vc x,t; 1)$ denote the multivariate CGF corresponding to $(Z(t_1),\ldots,Z(t_N))$ conditional on $Z(0)=a$, where $\vc x=(x_1,\ldots,x_N)$ and $e^{\vc x}=(e^{x_1},\ldots,e^{x_N})$. The first and second partial derivatives of $K(\vc x,t;  a)$ are  \begin{eqnarray*}K'(\vc x,\vc t; a)&:=&\left(\frac{\partial K(\vc x, \vc t; a)}{\partial x_1},\ldots,\frac{\partial K(\vc x,\vc t; a)}{\partial x_N}\right)^\top,\\
K''(\vc x,\vc t;a)&:=&\left(\frac{\partial^2 K(\vc x,\vc t; a)}{\partial  x_i\partial x_j}: i,j=1,\ldots, N\right),\end{eqnarray*}and can be expressed explicitly. 
The saddlepoint approximation to $p_{\vc k}(\vc t;\vc a)$ for $\vc k>\vc 0$ is 
\begin{equation}\label{SPA_mult} \tilde{p}_{\vc k}(\vc t; a)=\dfrac{1}{(2\pi )^{N/2}|K''(\tilde{\vc x},\vc t;  a)|^{1/2}}\exp\left\{K(\tilde{\vc x},\vc t;  a)-\tilde{\vc x} \vc k\right\},\end{equation}where the saddlepoint $\tilde{\vc x}=\tilde{\vc x}(\vc k,\vc t;a)$ is the unique solution to the $N$-dimensional system of saddlepoint equations
\begin{equation}\label{saddle_equ_multi}
K'(\tilde{\vc x},\vc t; a)=\vc k
\end{equation} in the maximal convergence set $\mathcal{S}(t)$ of $g(e^{\vc x},\vc t;1)$ in a neighbourhood of $\vc 0\in \mathbb{R}^2$.
No explicit solution exists for $\tilde{\vc x}$, which must be evaluated numerically either by minimizing the exponent in \eqref{SPA_mult} or by solving \eqref{saddle_equ_multi}. If $\vc k=(k_1,\ldots,k_i,0,0,\ldots,0)$ for some $1<i\leq N-1$, then we write ${p}_{\vc k}(\vc t; a)= {p}_{\tilde{\vc k}}(\vc t; a)\, p_{0}(\tau_{i+1};k_i)$ where $\tilde{\vc k}=(k_1,\ldots,k_i)>\vc 0$, and we approximate ${p}_{\tilde{\vc k}}(\vc t; a)$ using \eqref{SPA_mult}.


Similar to the univariate case, the saddlepoint approximation error decreases as $a$ increases,
\begin{equation}\label{err_spa_mult}{p}_{\vc k}(\vc t;a)=\tilde{p}_{\vc k}(\vc t;a)\{1+\mathcal{O}(1/a)\},\end{equation}leading to the following lemma.

\begin{lemma}The error in the saddlepoint approximated log-likelihood is
\begin{eqnarray*}{\ell}(\lambda,\mu; \vc t,\vc k)-\tilde{\ell}(\lambda,\mu; \vc t,\vc k)&=& \mathcal{O}\left\{\left(\min_{1\leq i\leq M}\{k_{i,0}\}\right)^{-1}\right\}.\end{eqnarray*}
\end{lemma}
\begin{proof}Following the same argument as in the proof of Lemma \ref{lem_err_1}, we have
\begin{eqnarray*}{\ell}(\lambda,\mu; \vc t,\vc k)-\tilde{\ell}(\lambda,\mu; \vc t,\vc k)&=&\sum_{i=1}^M \log {p}_{\vc k_{i}}(\vc t;k_{i,0})-\log \tilde{p}_{\vc k_{i}}(\vc t;k_{i,0})
\\&=&\sum_{i=1}^M \mathcal{O}(1/k_{i,0})=\mathcal{O}\left\{\left(\min_{1\leq i\leq M}\{k_{i,0}\}\right)^{-1}\right\}.\end{eqnarray*}

\end{proof}Since $$\mathcal{O}\left\{\left(\min_{1\leq i\leq M}\{k_{i,0}\}\right)^{-1}\right\}\leq \mathcal{O}\left\{\left(\min_{1\leq i\leq M,1\leq j\leq N,}\{k_{i,j-1}\}\right)^{-1}\right\},$$ the overall error in the saddlepoint approximated log-likelihood is generally smaller in the multivariate case than in the univariate case. Strict inequality is achieved as soon as the population size at the first observation time is strictly larger than at a subsequent observation time in at least one trajectory.
Depending on the observed trajectories, there could be a non-negligible gain in using the multivariate approach if the trajectories tend to decrease from their initial state, which is more likely to happen in subcritical cases. However, the drawback of the multivariate approach is the loss of tractability of the saddlepoint and the numerical errors resulting from its numerical evaluation.

Table \ref{tmult} compares the relative error between the SPMLE and the true MLE for $\lambda$ and $\mu$ using the univariate and the multivariate saddlepoint approaches, on some examples of single trajectories ($M=1$) observed at equidistant intervals with different initial population sizes $a$ and observed population size vectors $\vc k$. 
We see that the theoretical gain in approximation error is not always reflected in the numerical experiments, which may be due to a large multiplying constant in the error term of maximal order in the multivariate case.

\begin{table}
\centering
\begin{tabular}{cc|cc|cc}
$a$&$\vc k$& RE $\lambda$ uni&RE $\lambda$ multi &RE $\mu$ uni & RE $\mu$ multi
\\\hline
20 &   [13,     7 ,    6 ,    2 ,    5] & 0.0760  &  \textbf{0.0753}&0.0486 &   \textbf{0.0451}\\
10  &  [10  ,  20 ,   33,    67,    80]& 0.0217  &  \textbf{0.0216}& 0.0316  &  \textbf{0.0314}\\
30  &  [11 ,    7 ,    3   ,  5  ,   5] & \textbf{0.1134}  &  0.1520& 0.0605  &  \textbf{0.0405}\\
10   &  [6  ,   3 ,    7,     7 ,    3] &0.0870  &  \textbf{0.0851}& \textbf{0.0669}   & 0.0670\\
20  &  [16 ,   16,    10,     5 ,    8] & \textbf{0.0353} &   0.1091& \textbf{0.0258 } &  0.0835
\end{tabular}
\caption{\label{tmult}Relative error (RE) between the SPMLE and the true MLE for $\lambda$ and $\mu$ using the univariate (uni) and the multivariate (multi) saddlepoint approaches, for different initial population sizes $a$ and observed population size vectors $\vc k$. The lower errors are highlighted.}
\end{table}

Similar to the univariate case, a Gaussian approximation to the multivariate distribution of the population sizes can also be derived, and the corresponding log-likelihood can be evaluated efficiently \cite{ross2006parameter}. 


\end{document}